\documentclass[12,reqno]{amsart}
\usepackage{amsfonts}
\usepackage{amssymb}
\usepackage{a4wide}
\usepackage{pgf,pgfarrows,pgfautomata,pgfheaps,pgfnodes,pgfshade}
\usepackage{url}
\usepackage{bbm}
\usepackage{graphicx}

\usepackage[numbers]{natbib}

\usepackage{hyperref}
\numberwithin{equation}{section}

\newcommand{\diff}{\,\mathrm{d}}
\newcommand{\diffns}{\mathrm{d}}
\newcommand{\hata}{\hat\alpha}

\newtheorem{defi}{Definition}[section]
\newtheorem{thm}[defi]{Theorem}
\newtheorem{lemm}[defi]{Lemma}
\newtheorem{remark}[defi]{Remark}

\newtheorem{assum}[defi]{Assumption}
\newtheorem{prop}[defi]{Proposition}
\newtheorem{prob}[defi]{Problem}

\newcommand{\R}{\mathbb{R}}

\newcommand{\E}{\mathbb{E}}
\newcommand{\PP}{\mathbb{P}}

\DeclareMathOperator{\sgn}{sgn}

\begin{document}

	\title[Optimal Control for Jump Diffusion with Singular Drifts]{Stochastic Optimal Control for Jump Diffusion Models with Singular Drifts}
	
	\author[E. F. Kameh]{Edward Fuituh Kameh}
	\address{Department of Mathematics and Computer Science, University of Bamenda, Cameroon }
	\email{fekkameh@gmail.com}
	
	\author[A.-M. Bogso]{Antoine-Marie Bogso}
	\address{Faculty of Sciences, Department of Mathematics, University of Yaounde I,
		P.O. Box 812, Yaounde, Cameroon}
	\email{antoine.bogso@facsciences-uy1.cm}

	\author[O. Menoukeu Pamen]{Olivier Menoukeu Pamen}
	\address{IFAM, Department of Mathematical Sciences, University of Liverpool, L69 7ZL, United Kingdom}
	\email{menoukeu@liverpool.ac.uk}
	
	\author[F. C. Shu]{Felix Che Shu}
	\address{Department of Mathematics and Computer Science, University of Bamenda, Cameroon}
	\email{shufche@gmail.com}
	
	\date{ }
	
	
	
	\maketitle
	
	\begin{abstract} 
		We study a stochastic optimal control problem for jump-diffusion systems whose drift coefficient is piecewise Lipschitz continuous and exhibits threshold-induced discontinuities. Such dynamics naturally arise in applications with interventions triggered by prescribed safety levels, notably in insurance surplus management with dividend payments and capital injections. We establish both necessary and sufficient optimality conditions for this class of control problems. Our approach combines an explicit Sobolev-type representation of the first variation process that accounts for jumps with smooth approximation and Ekeland's variational principle. A key analytical ingredient is a density estimate for the state process, which yields time-integrated convergence estimates needed to pass to the limit in the first variations and adjoint processes. As an application, we study optimal premium adjustment and reserve management for an insurer whose surplus dynamics incorporate threshold-based dividend payments and capital injections.
		
	\end{abstract}

	\section{Introduction}
	Stochastic optimal control provides a framework for modelling
	and optimising systems subject to uncertainty, with applications in finance, engineering, and the natural sciences. The aim is to determine a control process that minimises an expected cost or maximises an expected reward for a system whose dynamics are governed by stochastic differential equations (SDEs). In the classical diffusion setting, the controlled system satisfies
	\begin{equation*}
		\mathrm{d}X_t^{\alpha} = b(t, X_t^{\alpha}, \alpha_t)\mathrm{d}t + \sigma(t, X_t^{\alpha}, \alpha_t)\mathrm{d}B_t, 
		\quad X_0^{\alpha} = x_0,
	\end{equation*}
	where $B_t$ is a standard Brownian motion, $\alpha_t$ is an admissible control, and $b$ and $\sigma$ denote the drift and diffusion coefficients, respectively. The objective is to characterise the optimal control $\hat{\alpha}$ minimising
	\begin{equation*}
		J(\alpha) = \mathbb{E}\big[\int_0^T f(t, X_t^{\alpha}, \alpha_t)\mathrm{d}t + g(X_T^{\alpha})\big].
	\end{equation*}
	
	Two main approaches to stochastic optimal control are the	\emph{dynamic programming method}, which characterises the value	function through a Hamilton-Jacobi-Bellman (HJB) equation, and the	\emph{stochastic maximum principle} (SMP), which formulates optimality conditions through adjoint processes and Hamiltonian variational inequalities. The SMP was first introduced by Kushner \cite{Kush72} and further developed in \cite{Ben81,Bis78,Haus86,Pen90,YZ99}. It provides necessary optimality conditions through adjoint processes and variational inequalities. Under suitable concavity or convexity assumptions, these conditions are also sufficient. While dynamic programming typically exploits a Markovian formulation and an appropriate solution concept for the HJB equation, the SMP analyses perturbations of controlled trajectories without requiring differentiability of the value function.
		
	Traditional SMP results typically impose differentiability assumptions on the coefficients, together with suitable growth and Lipschitz conditions; see, for example,~\cite{Pen90,YZ99}, which also consider non-convex control domains, and the references therein. The literature also includes results for jump-diffusion systems, notably the work of Tang and Li \cite{TangLi1994}; see also \cite{OS07}.
		
		These regularity assumptions ensure existence and
		uniqueness of strong solutions and justify differentiation
		of the state process with respect to control perturbations.
		In many real-world systems, however, such regularity is
		too restrictive: dynamics may exhibit discontinuities
		or abrupt changes, as in hybrid systems, threshold-type
		financial models, switching diffusions, or controlled
		processes with regime changes.
		When the drift is only piecewise Lipschitz, classical
		variational methods cannot be applied directly because
		the first variation process may not exist in the
		classical sense.
		These challenges motivate the development of an SMP
		under weaker regularity assumptions.

Recent studies have established strong solvability and regularity results for Brownian driven SDEs with irregular drifts; see, for example,~\cite{krylov2005strong,Krylov20211,Krylov20212,MMNPZ13, MenTan19,MNP2015,XZhang2005,XZhang2011} and the references therein. For jump-diffusion equations, Przyby{\l}owicz and Sz\"olgyenyi \cite{PrzSzo21,PrSX21} have investigated well-posedness and numerical approximation under piecewise Lipschitz drift assumptions.

Stochastic control problems with nonsmooth coefficients have also been studied through PDE-based approaches.	De Feo \cite{Defeo26} considers infinite-horizon discounted control problems with measurable coefficients and locally uniformly elliptic diffusion. Using $L^p$-viscosity solutions with local Sobolev regularity, the author establishes verification theorems and, under additional hypotheses, constructs optimal feedback controls and characterises the value function as the unique solution of the associated HJB equation.
		Criens \cite{criens2025} considers controlled diffusions with measurable coefficients, uniformly elliptic diffusion and a drift dominated by an $L^d$ function, within a relaxed-control framework that does not require uniqueness of the controlled state equations. The results connect exit-time and finite-horizon value functions	with elliptic and parabolic HJB equations in the $L^p$-viscosity sense, provide quantitative regularity estimates, and identify a nonlinear semigroup exhibiting a regularisation-by-noise effect.	Du and Wei \cite{DuWei26} study finite-horizon control problems with additive nondegenerate Brownian noise, where the control-dependent	drift and running cost are merely measurable in the state variable and satisfy global $L^p$-integrability assumptions. Combining PDE estimates with policy iteration, they identify the value function with the unique $W^{1,2}_p$ solution of the HJB equation and establish a verification theorem and the dynamic programming principle. They further show that mollification need not yield convergence of the corresponding value functions, and prove locally uniform convergence when the control set is countable.
		
		A complementary line of research develops Pontryagin-type stochastic maximum principles for Brownian-driven systems with irregular drifts \cite{BKPMarx,MenTan22,TBMP24}. In particular, the authors of \cite{BKPMarx} consider drifts of the form $b(t,x,a)=b_1(t,x)+b_2(x)b_3(t,a)$, where $b_1$ is bounded and Borel measurable, $b_2$ has bounded variation, and $b_3$ is bounded	and smooth in the control variable. Their approach combines Malliavin-Sobolev regularity, a representation of the first variation	through space-time local time, smooth approximation and Ekeland's variational principle to derive necessary and sufficient optimality conditions.

The present paper develops this variational approach for finite-activity jump-diffusion control systems with piecewise Lipschitz drifts exhibiting threshold discontinuities. Its main novelty is the establishment of necessary and sufficient stochastic maximum principles in the simultaneous presence of discontinuous drifts and jumps. Unlike the PDE-based approaches discussed above, we derive optimality conditions through first variations and adjoint processes, without requiring differentiability of the value function	or establishing solvability of the associated nonlocal HJB equation. The contribution is therefore complementary to those approaches, which permit more general coefficient structures in other respects.

The proof combines a change of variables based on the transformation $G$ (see proof of Lemma \ref{prop:convXn0}), an explicit representation of the
		Sobolev first variation that accounts for jumps, an
		approximation argument, and Ekeland's variational principle.
		The transformation yields globally Lipschitz coefficients
		and provides the basis for constructing the approximating
		equations. A central difficulty is to justify convergence of coefficient approximations evaluated along approximating state processes. In the additive Brownian setting \cite{BKPMarx,MenTan22}, Girsanov's theorem permits a reduction to expectations involving Brownian motion, for which explicit Gaussian	estimates are available. A change of measure does not remove the jumps, so this reduction cannot be applied directly here. Instead, we establish a time-integrated convergence result using uniform moment bounds and a density estimate of order $t^{-1/2}$; see Theorem \ref{thm:global_bound_clean}. This result gives the stability needed to pass to the limit in the first variations and	adjoint processes; see Lemma~\ref{lem:conv.y.phi}.
	
	Our insurance application illustrates the relevance of this framework. Dividend payments, capital injections and premium adjustments may change when reserves cross prescribed thresholds, producing discontinuous yet piecewise Lipschitz drifts. The surplus model in Section \ref{motivexam} illustrates how such mechanisms interact with jump risk and motivates the optimality conditions developed in this paper.
	
	
	The remainder of the paper is organised as follows. Section~\ref{motivexam} introduces the motivating example, formulates the general control problem, and establishes regularity results for the state equation. Section~\ref{secSMP} presents the necessary and sufficient optimality conditions. Finally, Section \ref{secappl} applies these results to an insurance model, analysing optimal premium adjustment and reserve management under jump risk.

	\section{A motivating example and preliminary results}\label{motivexam}

	\subsection{Motivating example and problem formulation}\label{sectappsub1}
	
	Insurance companies manage liabilities subject to random claims, leading to stochastic fluctuations in surplus. Controlling the surplus requires balancing premium income, reserves, and regulatory constraints. This leads to a stochastic control problem in which the insurer seeks an optimal premium or reserve policy to ensure stability while limiting risk.
	
	
	\subsubsection{Motivating example} We consider an insurance company whose liability process (or payment function) at time $t$ is denoted by $L_t$, representing the total amount of insurance claims minus the premiums received over the interval $[0, t]$. We model the liability using a classical compound Poisson risk model, possibly perturbed by diffusion or jump-diffusion terms. Specifically, the liability of the company at time $t$ is described by the following SDE:
	
	\begin{equation}
		-\mathrm{d} L_t = (b_t + \alpha_t) \, \mathrm{d}t + \sigma_t \, \mathrm{d}B_t + \mathrm{d} \big( \sum_{k=1}^{N_t} Z_k \big), 
		\quad t \ge 0, \quad L_0 = l, \label{liability}
	\end{equation}
	where $b_t \geq  0$ is the liability rate representing the expected liability (or gain) per unit time due to premium loading, $\alpha_t$ is the premium rate (or premium policy) which acts as the control variable, and $\sigma_t > 0$ is the volatility measuring liability risk. We assume that the insurer is not allowed to invest in risky assets due to regulatory restrictions.
	
	At the initial time $t=0$, the insurer deposits an amount $X_0$ to cover potential future claims exceeding premiums. Let $X_t$ denote the insurer's cash balance at time $t$, which is composed of the initial capital minus net outflows up to time $t$, all accumulated at the compound interest rate. This can be written as
	
	\begin{equation}
		X_t = e^{\Delta_t} \big( X_0 - \int_0^t e^{-\Delta(s)} \, \mathrm{d} L_s \big), 
		\quad X_0 = x, \label{cbal}
	\end{equation}
	where $\Delta_t = \int_0^t \delta_s \, \mathrm{d}s$ and $x \ge 0$ represents the initial reserve. Applying It\^o's formula, the cash balance $X_t$ satisfies the following SDE:
	\begin{equation}
		-\mathrm{d} X_t = (\delta_t X_t + b_t + \alpha_t) \, \mathrm{d}t + \sigma_t \, \mathrm{d}B_t + \mathrm{d} \big( \sum_{k=1}^{N_t} Z_k \big), 
		\quad t \in [0,T], \quad X_0 = x, \label{cbalsde}
	\end{equation}
	where $\alpha_t$ is the control variable. Further, assume that when the surplus exceeds a positive threshold $H$, the insurer may withdraw a fixed amount (dividend or fee), and when the surplus falls below $-H$, the insurer injects a fixed amount (or reduces net outflows). Incorporating this mechanism, the SDE \eqref{cbalsde} becomes:
	
	\begin{eqnarray*}
		\begin{cases}
			-	\mathrm{d} X_t = \big(\delta_t X_t + b_t + \alpha_t - \beta \, \mathrm{sgn}(X_t) \mathbf{1}_{\{|X_t| > H\}}\big) \, \mathrm{d}t 
			+ \sigma_t \, \mathrm{d}B_t + \mathrm{d} \big( \sum_{k=1}^{N_t} Z_k \big), 
			\quad t \in [0,T], \\
			\quad X_0 = x, 
		\end{cases}
	\end{eqnarray*}
	where $\beta > 0$ represents the magnitude of the fixed drain or injection. The above model  maintains the surplus within a target corridor, preventing large oscillations, while also accounting for emergency capital infusion or premium adjustments to avoid ruin.
	The above equation can be rewritten as:
	\begin{align}\label{eqmot1}
		\diffns X^{\alpha}_t 
		=&-\big\{\delta_t X^{\alpha}_t + b_t + \alpha_t - \beta \, \mathrm{sgn}(X^{\alpha}_t) \mathbf{1}_{\{|X^{\alpha}_t| > H\}}\big\}\diffns t - \sigma_t \diffns B_t -\int_{\mathbb{R}\backslash \{0\}}z N(\mathrm{d}z,\mathrm{d}t).
	\end{align}
	We assume that the fund manager's objective is to determine the optimal premium rate $\alpha$ that minimises the second moment of the fund value. In other words, the goal is to find $\alpha$ that keeps $X$ as close as possible to zero. More formally, 
	\begin{equation}\label{eqmot2}
		\mathbb{E}\big[\big(X^{\hat \alpha}_T\big)^2\big]=\inf_{\alpha \in \mathcal{A}}\mathbb{E}\big[(X^{\alpha}_T)^2\big],
	\end{equation}
	where $\hat \alpha$ is an optimal control (if it exists). Although the drift term in equation \eqref{eqmot1} is discontinuous at $H$, it can still be seen as a piecewise Lipschitz function. However, the problem of optimal control for jump-diffusion with piecewise Lipschitz coefficients lies outside the scope of existing literature on stochastic optimal control. 
	
	
	\subsubsection{General Problem Formulation}
	

	Motivated by the above example, we consider the following general controlled jump--diffusion system:
	\begin{equation}
		\label{eqconpb1}
		X^\alpha_t
		= x + \int_0^t b(X^\alpha_u,\alpha_u)\,\mathrm{d}u
		+ \sigma B_t
		+ \int_0^t\int_{\mathbb{R}\backslash \{0\}}
		\gamma(z)\,N(\mathrm{d}z,\mathrm{d}s),
	\end{equation}
	where $b$ is a deterministic drift function,
	$\sigma\in\mathbb{R}^{1\times d}$ is a constant vector, and
	$\gamma:\mathbb{R}\setminus\{0\}\to\mathbb{R}$ is a measurable
	function specifying the jump amplitude associated with a mark $z$.
	Throughout the paper, we assume that the Brownian motion $B$ and
	the Poisson random measure $N$ are independent and defined on a probability space $(\Omega,\mathcal{F},\mathbb{P})$ equipped with the completed filtration $(\mathcal{F}_t)_{t\in[0,T]}$. The intensity measure of $N$ is $\nu(\mathrm dz)\mathrm dt$, where
	$\nu(\mathbb R_0)<\infty,
	\,\,\mathbb R_0:=\mathbb R\setminus\{0\}.$ Moreover, the jump coefficient $\gamma:\mathbb R_0\to\mathbb R$
	is measurable and satisfies
	\begin{equation}\label{eq:jump_exp_moments}
		\int_{\mathbb R_0}
		\bigl(e^{a_0|\gamma(z)|}-1\bigr)\,\nu(\mathrm dz)<\infty,
		\quad \text{for some }a_0>0,
	\end{equation}
	where the required lower bound on $a_0$ is specified in the
	proof of Theorem \ref{thmmainresder1}. The above bound also implies
	\begin{equation}\label{eq:jump_moments}
		\int_{\mathbb R_0}|\gamma(z)|^p\,\nu(\mathrm dz)<\infty,
		\quad \text{for every }p\ge2.
	\end{equation}
	
	The control variable $\alpha:[0,T]\times \Omega \mapsto \mathbb{A}\subset \mathbb{R}$ is $\mathcal{B}([0,T])\otimes \mathcal{F}$-measurable, and $({\mathcal F}_t)_{t\in[0,T]}$-adapted. $\mathbb{A}\subseteq \mathbb{R}$ is a closed convex subset of $\mathbb{R}$. The criterion to be optimised over the set of admissible controls $\mathcal{A}$ (see Definition \ref{defadminC1}) has the following form.
	\begin{align} 
		J(\alpha):= \mathbb{E}\big[ \int_0^T  f(s,X^\alpha_s,\alpha_s)\diff s + g(X^\alpha_T)\big],\label{perfunct1}
	\end{align}
	where $f$ and $g$ are function satisfying some conditions.
	The optimal control problem is the following
	\begin{prob}\label{pb1}
		Find $\hat \alpha \in \mathcal{A} $ such that 
		\begin{align} \label{eqconpbval1}
			J(\hat \alpha)=	\sup_{\alpha \in \mathcal{A}}J(\alpha).
		\end{align}
\end{prob}
\begin{defi}\label{defadminC1}
	The set of admissible controls given in Problem \ref{pb1} is defined as:
	\begin{multline*}
		\mathcal{A} := \big\{\alpha:[0,T]\times \Omega\to \mathbb{A}, \text{ progressive, \eqref{eqconpb1}}\text{ has a unique strong solution and }\\ \E\sup_{t\in[0,T]}|\alpha(t)|^{4}<\infty\big\}.
	\end{multline*}
\end{defi}
The following definition corresponds to \cite[Definition 2.1]{LeoSzo17}.
\begin{defi}[piecewise Lipschitz function]
	Let $I \subset \mathbb{R}$ be an interval. A function $f : I \to \mathbb{R}$ is said to be \emph{piecewise Lipschitz} if there exist finitely many points $\xi_1 < \cdots < \xi_m \in I$ such that $f$ is Lipschitz continuous on each of the intervals $
	(-\infty, \xi_1) \cap I, \,\, (\xi_m, \infty) \cap I$ and $ (\xi_k, \xi_{k+1}), \quad k = 1, \dots, m-1.$
\end{defi}
We assume the following conditions on $b,f$ and $g$.

\begin{assum}\label{mainassum2}

	\leavevmode
	$$
	b(x,a)=b_1(x,a)+b_2(x)+b_3(x),
	$$
	where 
	\begin{itemize}
		\item $b_1$ is globally bounded and $\mathcal{C}^1$	in its first and second variables, with uniformly bounded derivatives
		\item $b_2$ is uniformly bounded and piecewise Lipschitz with $k \in \mathbb{N}$ discontinuities in the points $\xi_1,\ldots,\xi_k \in \mathbb{R}$
		
		\item $b_3$ is $\mathcal{C}^1$ with bounded derivative and satisfies a linear growth condition.
		
		
		\item $f$ and $g$ are continuously differentiable in their arguments;
		\item there exists $C>0$ such that
		\[
		|f(t,x,a)| + |\partial_x f(t,x,a)| + |\partial_a f(t,x,a)|
		\le C(1+|x|^2+|a|^2), \quad \text{for all } (t,x,a)
		\]
		and
		\begin{equation*}
			|g(x)| + |\partial_xg(x)| \le C(1 + |x|^2).
		\end{equation*}
	\end{itemize}
\end{assum}

\begin{remark}

			The scalar-control assumption is made for notational
			simplicity. Controls in a closed convex subset of
			$\mathbb R^m$ can be treated under corresponding
			differentiability and growth assumptions, using gradients
			and scalar products in the optimality conditions.
			
			The one-dimensional state assumption is central to the
			explicit derivative-free representation of the first
			variation used here. For multidimensional states, an
			alternative would be to develop an approximation-based
			SMP without relying on such a representation.
			A possible approach is to combine a suitable
			multidimensional transformation \cite{LeoSzo16} with
			a nonsmooth maximum principle \cite{BDM07}.
			Establishing the required stability and convergence
			estimates, initially for diffusion systems without jumps,
			is beyond the scope of this paper.
		
\end{remark}
Note that the controlled SDE \eqref{eqmot1} in the motivating example is a particular case of \eqref{eqconpb1}.  
In the subsequent section, we provide results on existence, uniqueness, and regularity for the state equation~\eqref{eqconpb1}, laying the foundation for establishing Pontryagin-type necessary and sufficient conditions under piecewise Lipschitz drift coefficients.


\subsection{Regularity of SDE with piecewise Lipschitz coefficient with jumps}
In this section, we consider the SDE

\begin{equation}	\label{eq:SDErvmain1}
X^x_t = x + \int_0^t b(X^x_u)\mathrm{d}u + \sigma B_t +\int_0^t\int_{\mathbb{R}_0}\gamma (z)N(\mathrm{d}z,\diffns s),
\end{equation}
where $b$ is piecewise Lipschitz with $k \in \mathbb{N}$ discontinuities in the points $\xi_1,\ldots,\xi_k \in \mathbb{R}$ and $b$ also satisfies linear growth condition. More precisely $b$ satisfies	
\begin{assum}\label{mainAssum}
\leavevmode
$$
b=b_1+b_2,
$$
where 
\begin{itemize}
	\item $b_1$ is Lipschitz and satisfies a linear growth condition;
	\item $b_2$ is uniformly bounded and piecewise Lipschitz with $k \in \mathbb{N}$ discontinuities in the points $\xi_1,\ldots,\xi_k \in \mathbb{R}$		
\end{itemize}
\end{assum}
It is well known from \cite[Theorem 4.1]{PrSX21} that the SDE \eqref{eq:SDErvmain1} has a unique strong solution. In this section we want to show that the solution is Sobolev differentiable with respect to its initial condition and the derivative has an explicit representation.

To this end, let $(X^{x,n})_{n\geq1}$ be the sequence of strong	solutions to the following approximating equations
\begin{equation}\label{eq:SDErvmainapp1}
	X_t^{x,n}
	=x+\int_0^t b_n(X_u^{x,n})\,\mathrm{d}u+\sigma B_t
	+\int_0^t\int_{\mathbb R_0}
	\gamma(z)\,N(\mathrm{d}z,\mathrm{d}s),
\end{equation}
where $b_n:=b_{1,n}+b_{2,n}$ is the sequence constructed in
Lemma \ref{lem:barb_bn_simultaneous} in the appendix.

Existence and uniqueness of the strong solution $X^{n,x}$
follow from classical results on SDEs with jumps
(see \cite{Apple2004,KunitaB19}).
Since $b_n$ is smooth with bounded first and second
derivatives, \cite[Chapter V, Theorem 39, p. 305]{Prot05}
implies that $X^{n,x}$ is continuously differentiable
with respect to the initial condition, and we have 
\begin{equation}\label{eq:SDEderX1n}
\frac{\partial X_t^{n,x}}{\partial x}
=
1+\int_0^t b_n'(X_u^{n,x})
\frac{\partial X_u^{n,x}}{\partial x}\,\mathrm du
=
\exp\big\{
\int_0^t b_n'(X_u^{n,x})\,\mathrm du
\big\},
\end{equation}
where $b_n'$ denotes the derivative of $b_n$ and the second equality follows by solving the linear integral
equation satisfied by $\frac{\partial X_t^{n,x}}{\partial x}$.
In the next Lemma, we show that the sequence $(X^{n,x}_{t})_{n \in \mathbb{N}}$ converges strongly in $L^2(\Omega,\mathbb{P})$.

\begin{lemm}
\label{prop:convXn0}
Let $b$ be as in Assumption \ref{mainAssum}, and let
$(b_n)_{n\ge 1}$ be given as above. Let $X^{n,x}_\cdot$ denote the strong solution of	\eqref{eq:SDErvmainapp1}, and let $X^x_\cdot$ denote a strong solution of \eqref{eq:SDErvmain1},
	with respective drifts $b_n$ and $b$. Then, the sequence $\{X^{n,x}_t\}_{n \ge 1}$ satisfies
$$
X^{n,x}_t \rightarrow X^x_t
\quad \text{strongly in } L^2(\Omega,  \mathbb{P})
\quad \text{as } n \to \infty.
$$

\end{lemm}

\begin{proof}
See Appendix \ref{appen}
\end{proof}

The following result also holds true.

\begin{thm}\label{thmmainresder1}
Under Assumption \ref{mainAssum}, the first variation process (in the Sobolev sense) of the strong solution $X^x_t$, to the SDE \eqref{eq:SDErvmain1} exists and has the
following explicit representation
\begin{align}\label{eqfirstvarX1}
	\frac{\partial X^x_t}{\partial x}=&\exp\big\{\int_0^tb_1'(X^x_u) \mathrm{d}u+2\big(\tilde b_2(X^x_t)- \tilde b_2(x)-	\int_0^tb_1(X^x_u)b_2(X^x_u)  \mathrm{d}u	\notag\\
	&-	\int_0^tb_2^2(X^x_u) \mathrm{d}u	-\int_0^tb_2(X^x_u) \mathrm{d}B_u\notag\\
	&-\int_0^t\int_{\mathbb{R}_0}\big\{ \tilde b_2(X^x_u+\gamma(z))-\tilde b_2(X^x_u)\big\} N(\mathrm{d}z,\mathrm{d}u)\big)\big\},
\end{align}
where  
$
\tilde b_2(x)=\int_{0}^x b_2(y)\diffns y.
$
In addition, for every $p\ge 1$, there exists a constant $C>0$ depending on the Lipschitz constant of $b_1$,  $\|b_2\|_{\infty}, T, p$, such that
\begin{align}	\label{eqderX1}
	\mathbb{E}\big[\big|\frac{\partial X^{x}_t}{\partial x}\big|^p\big] \leq C.
\end{align}
\end{thm}
\begin{proof} See Appendix \ref{appen}.

\end{proof}

\begin{remark}\label{remextdiff1}	The
	constant-diffusion assumption is not essential to the
	antiderivative argument used in the proof of the
	first-variation representation. Indeed, for a state-dependent diffusion coefficient $\sigma\in C_b^1(\mathbb R)$	satisfying $\inf_{x\in\mathbb R}|\sigma(x)|>0$, define
	\[
	F_n(x):=\int_0^x\frac{b_{2,n}(y)}{\sigma^2(y)}\,\mathrm{d}y.
	\]
	Then
	\[
	\frac12\sigma^2(x)F_n''(x)
	=\frac12b_{2,n}'(x)
	-\frac{b_{2,n}(x)\sigma'(x)}{\sigma(x)}.
	\]
	Thus It\^o's formula again expresses
	$\int_0^t b_{2,n}'(X_s^{n,x})\,\mathrm{d}s$ in terms that do not involve
	derivatives of $b_{2,n}$. For state-dependent diffusion, the
	first variation $J_t^n:=\partial_x X_t^{n,x}$ of the approximating solution satisfies 
	\[
	\mathrm{d}J_t^n=b_n'(X_t^{n,x})J_t^n\mathrm{d}t
	+\sigma'(X_t^{n,x})J_t^n\mathrm{d}B_t,
	\qquad J_0^n=1,
	\]
	since the jump coefficient is independent of the state. Hence,
	\[
	J_t^n
	=\exp\Big\{
	\int_0^t b_n'(X_s^{n,x})\mathrm{d}s
	-\frac12\int_0^t|\sigma'(X_s^{n,x})|^2\mathrm{d}s
	+\int_0^t\sigma'(X_s^{n,x})\mathrm{d}B_s
	\Big\}.
	\]
	The weighted primitive therefore eliminates $b_{2,n}'$ from this
	representation, at the expense of additional terms involving
	$\sigma$ and $\sigma'$. Passing to the limit would require
	corresponding uniform estimates and convergence arguments.
\end{remark}

\section{Stochastic Maximum Principle}\label{secSMP}

\subsection{Main results}
The main results of this work are the following necessary and sufficient conditions
\begin{thm}
\label{thm:necc}
Suppose that $b$ satisfies Assumption \ref{mainassum2}.
Let $\hat\alpha \in \mathcal{A}$ be an optimal control and let $X^{\hat\alpha}$ be the associated optimal trajectory.
Then the first variation process $\Phi^{\hat\alpha}$ of $X^{\hat\alpha}$ is well-defined and it holds 
\begin{equation}
	\label{eq:nec.cond}
	\partial_{a}H(t, X^{\hat\alpha}_t,P^{\hat\alpha}_t,\hat\alpha_t )\cdot(\beta - \hat\alpha_t) \ge 0 \quad \PP\otimes \diff t\text{-
	a.e. for all } \beta \in \mathcal{A},
\end{equation} 
where $P^{\hat\alpha}$ is the adjoint process given by
\begin{equation}
	\label{eq:adj.proc}
	P^{\hat\alpha}_t := \E\Big[\Phi^{\hat\alpha}_{t,T} \partial_xg( X^{\hat\alpha}_T) + \int_t^T\Phi^{\hat\alpha}_{t,s} \partial_xf(s, X^{\hat\alpha}_s, \hat\alpha_s)\mathrm{d}s\mid \mathcal{F}_t \Big],
\end{equation}	
with $\Phi^{\alpha}$ given by 
\begin{align*} 
	\Phi^\alpha_{t,s}  =&\exp\big\{\int_t^s\partial_xb_1(X^{\alpha,x}_u,\alpha_u) \mathrm{d}u+\int_t^s\partial_xb_3(X^{\alpha,x}_u) \mathrm{d}u\\
	&+2\big(\tilde b_2(X^{\alpha,x}_s)- \tilde b_2(X^{\alpha,x}_t)-	\int_{t}^sb_2^2(X^{\alpha,x}_u) \mathrm{d}u-\int_{t}^{s}b_2(X^{\alpha,x}_u) \mathrm{d}B_u	\notag\\
	&-\int_t^sb_1(X^{\alpha,x}_u,\alpha_u)b_2(X^{\alpha,x}_u)  \mathrm{d}u\\
	&-\int_t^s\int_{\mathbb{R}_0}\big\{ \tilde b_2(X^{\alpha,x}_{u-}+\gamma(z))-\tilde b_2(X^{\alpha,x}_{u-})\big\} N(\mathrm{d}z,\mathrm{d}u)\big)\big\},
\end{align*}
where  
$
\tilde b_2(x)=\int_{0}^x b_2(y)\diffns y.
$
\end{thm}
\begin{proof}
See Section \ref{proofmainres1}. 
\end{proof}

\begin{remark}
Assumption \ref{mainAssum} can be seen as a reduced form of Assumption \ref{mainassum2} under feedback controls. Indeed, if $a=a(t,x)$ is bounded and piecewise Lipschitz in $x$, then $x\mapsto b_1(x,a(t,x))$ is bounded and piecewise Lipschitz. Since $b_3$ is Lipschitz with linear growth, the sum $b_1(x,a(t,x))+b_3(x)$ satisfies the conditions on $b_1$ in Assumption \ref{mainAssum}. Hence, Assumption \ref{mainassum2} reduces to Assumption \ref{mainAssum} in this setting.
\end{remark}

\begin{remark}\label{remextdiff2}
	Beyond the first-variation representation discussed in
	Remark \ref{remextdiff1}, extending
	Theorem \ref{thm:necc} to state-dependent diffusion requires
	corresponding stability estimates.
	Theorem \ref{thm:global_bound_clean} already provides
	the density bound
	\[
	\sup_{x\in\mathbb R}
	\|\rho_t^x\|_{L^\infty(\mathbb R)}
	\le C_Tt^{-1/2},
	\qquad t\in(0,T],
	\]
	for state-dependent diffusion coefficients satisfying its
	hypotheses. This gives the density estimate needed in
	the time-integrated approximation arguments.
	However, the additional stochastic term
	$\sigma'(X_t)J_t\,\mathrm{d}B_t$ in the first-variation
	equation and the corresponding drift term
	$\sigma'(X_t)Q_t$ in the adjoint equation require further
	uniform moment estimates and stability arguments.
	A full extension of the maximum principle is therefore
	left for future work.
\end{remark}
	
	\begin{thm}
		\label{thm:suff}
		Let the conditions of Theorem \ref{thm:necc} be satisfied. Let $\hat \alpha \in \mathcal{A}$ and let $X^{\hat \alpha}$ be the associated trajectory and $P^{\hat{\alpha}}_t$ be the corresponding adjoint process given by \eqref{eq:adj.proc}. In addition, assume that $g$ and $(x,a)\mapsto H(t,x,P^{\hat{\alpha}}_t,a)$ are convex for all $t\in [0,T]$. Further assume that 
		\begin{equation}
			\label{eq:suff.con}
			\partial_a H(t,X_t^{\hat\alpha},P_t^{\hat\alpha},\hat\alpha_t)
			(\hat\alpha_t-a)\leq0,
			\quad \text{for all }a\in A, \quad \PP\otimes \diff t\text{-a.e.}
		\end{equation}
		Then, $\hat\alpha$ is an optimal control.
		
	\end{thm}

	\begin{proof}
		See Section \ref{proofmainres1}. 
	\end{proof}
	
	.
	
	\textbf{Proof strategy}. We now outline the strategy for proving Theorems \ref{thm:necc} and \ref{thm:suff}.
	
	\begin{enumerate}
		\item Step 1. Using Ekeland’s variational principle, we show that an optimal control for problem \eqref{eqconpb1} is also optimal for a suitably perturbed version of the approximating problem with performance functional $J_n(\alpha^n)$.
		
		\item Step 2. For the approximating problem with smooth drift coefficients, we derive a maximum principle involving the state process $X_n^{\hat\alpha^n}$ and its associated first variation process $\Phi^{\hat\alpha^n}_n$.
		
		\item Step 3. The final and most delicate step consists in passing to the limit as $n \rightarrow \infty$ and establishing a stability property of the maximum principle.
	\end{enumerate}
	
	We begin by addressing this limit step through a series of preparatory lemmas, which will be assembled to prove Theorem \ref{thm:necc} at the end of this section.
	\begin{lemm}
		\label{lem:conv.Xnn}
		We have the following bounds: 
		\begin{itemize}
			\item[(i)]For every sequence $(\alpha^n)_n$ in $\mathcal{A},\,\, p\geq 1$, it holds $\underset{n}{\sup}\E\big[\underset{t \in [0,T]}{\sup}|X^{n,\alpha^n}_t|^p \big]<\infty$.
			\item[(ii)] For every $\alpha^1,\alpha^2 \in \mathcal{A}$ it holds 
			$$ 
			\E\big[| X^{n,\alpha^1}_t - X^{\alpha^2}_t|^2 \big] \le C\big( d(\alpha^1,\alpha^2)^4 + \int_0^{T}\int_{\mathbb{R}}| b_{2,n}( x)- b_{2}(x)|^4\rho^{\alpha^2}_s(x)\diffns x\mathrm{d}s\big),
			$$ 
			where $\rho^{\alpha^2}_s$ is the density of $X^{\alpha^2}_s$ and $d$ is the distance defined by
			\begin{equation}\label{eqdist1}
				d(\alpha^1, \alpha^2) : = \E\big[\sup_{t \in[0,T]}|\alpha_t^1 - \alpha_t^2|^{4} \big]^{1/4} .
			\end{equation}
			\item[(iii)]Given $k \in \mathbb{N}$, for every sequence $(\alpha^n)_{n\geq 1}$ in $\mathcal{A}$ and $\alpha\in \mathcal{A}$ such that $d(\alpha^n,\alpha)\rightarrow 0$, it holds that 
			$$ \E\big[| X^{k,\alpha^n}_t - X^{k,\alpha}_t|^2 \big] \to 0 \text{ as } n \rightarrow \infty.$$
		\end{itemize}
	\end{lemm}
	\begin{proof}
		See Section \ref{proofmainres1}. 
	\end{proof}
	\begin{lemm}
		\label{lem:J.continuous}
		Let $\alpha\in \mathcal{A}$ and let $\alpha^n$ be a sequence of admissible controls satisfying $d(\alpha^n,\alpha)\to 0$.
		Then, we have
		\begin{itemize}
			\item[(i)] $| J_k(\alpha^n) - J_k(\alpha) | \to 0$ as $n\to \infty$  for every $k \in \mathbb{N}$ fixed. In other words, the function $J_k:(\mathcal{A},d) \to \mathbb{R}$ is continuous for each fix $k \in \mathbb{N}$.
			\item[(ii)] There exists a constant $C>0$ such that $|J_n(\alpha) - J(\alpha)| \le C\varepsilon_n$, with $\varepsilon_n\downarrow 0$.
		\end{itemize}
	\end{lemm}
	
	\begin{proof}
		See Section \ref{proofmainres1}. 
	\end{proof}

	We state the stability result for the first variation and of the adjoint processes.
	\begin{lemm}
		\label{lem:conv.y.phi}
		Let $\alpha\in \mathcal{A}$ and $\alpha^n$ be a sequence of admissible controls such that $d(\alpha^n,\alpha)\to 0$.
		Then, the processes $X^{\alpha^n}_n$ and $X^{\alpha}$ admit Sobolev differentiable flows denoted $\Phi^{\alpha^n}_n$ and $\Phi^{\alpha}$, respectively and for every $0\le t\le s\le T$ we have
		\begin{itemize}
			\item[(i)] $\E\big[|\Phi^{n,\alpha^n}_{t,s} - \Phi^\alpha_{t,s} |^2 \big] \to 0$ as $n\to \infty$,
			\item[(ii)] $\E\big[| P^{n,\alpha^n}_t - P^\alpha_t| \big] \to 0$ as $n\to \infty$,	where $P^\alpha$ is the adjoint process defined as
		\end{itemize}
		\begin{equation*}
			P^{\alpha}_t := \E\Big[\Phi^{\alpha}_{t,T} \partial_xg( X^{\alpha}_T) + \int_t^T\Phi^{\alpha}_{t,s} \partial_xf(s, X^{\alpha}_s, \alpha_s)\mathrm{d}s\mid \mathcal{F}_t \Big],
		\end{equation*}	
		and $P^{\alpha^n}_n$ is defined similarly, with $(X^{\alpha},\alpha, \Phi^\alpha)$ replaced by  $(X^{n,\alpha^n},\alpha^n, \Phi^{n,\alpha^n})$.
	\end{lemm}
	\begin{proof}
		See Section \ref{proofmainres1}. 
	\end{proof}

	\subsection{Proof of the Main results}\label{proofmainres1}

	\begin{proof}[Proof of Theorem \ref{thm:necc}]
		Let $\hat\alpha$ be an optimal control and fix $n\ge 1$.
		Under the linear growth assumptions on $f$ and $g$, the functional $J_n$ is bounded from above.
		By Lemma \ref{lem:J.continuous}, $J_n$ is continuous on $(\mathcal{A},d)$ and there exists $\varepsilon_n$ such that 
		\begin{equation*}
			J(\hat\alpha) - J_n(\hat\alpha)\le \varepsilon_n \text{ and } J_n(\alpha) - J(\alpha) \le \varepsilon_n\quad \text{for all } \alpha \in \mathcal{A}.
		\end{equation*}
		Hence, $J_n(\hat\alpha) \le \inf_{\alpha \in \mathcal{A}}J_n(\alpha) + 2\varepsilon_n$.
		Therefore, by the Ekeland's variational principle (see \cite{Ekeland79}), there exists a control $\hat\alpha^n \in \mathcal{A}$ such that $d(\hat\alpha, \hat\alpha^n)\le (2\varepsilon_n)^{1/2}$ and 
		\begin{equation*}
			J_n(\hat\alpha^n) \le J_n(\alpha) + (2\varepsilon_n)^{1/2}d(\hat\alpha^n,\alpha)\quad \text{for all}\quad \alpha \in \mathcal{A}.
		\end{equation*}
		Define the perturbed functional
		$$J^{\varepsilon_n}_n(\alpha):= J_n(\alpha) + (2\varepsilon_n)^{1/2}d(\hat\alpha^n,\alpha),$$
		so that the control $\hat\alpha^n$ is optimal for the problem with performance functional $J^{\varepsilon_n}_n$.
		
		Let $\beta \in \mathcal{A}$ be an arbitrary and fix $\varepsilon>0$. Set $\eta=\eta_n := \beta - \hat\alpha^n$. By the convexity of $\mathbb{A}$, we have  $\hat\alpha^n + \varepsilon\eta \in \mathcal{A}$.
		Thanks to the smoothness of $b_n$, the functional $J_n$ is G\^ateaux-differentiable, with derivative in the direction $\eta$ given by
		\begin{align*}
			\frac{\diffns}{\diffns\varepsilon}J_n(\hat \alpha^n + \varepsilon \eta)_{|_{\varepsilon = 0}}& = \mathbb{E}\big[\int_0^T\partial_xf(t, X^{n,\hat\alpha^n}_t, \hat\alpha^n_t)V^n_t + \partial_{\alpha}f(t, X^{n,\hat\alpha^n}_t, \hat\alpha^n_t)\eta_t\mathrm{d}t\\
			&\qquad + \partial_xg(X^{n,\hat\alpha^n}_T)V^n_T \big] ,
		\end{align*}
		where $V^n$ solves the corresponding linear equation
		\begin{equation*}
			\diffns V^n(t) = \partial_xb_n(t, X^{n,\alpha}_t,\alpha_t)V^n_t\mathrm{d}t + \partial_\alpha b_n(t, X^{n,\alpha}_t,\alpha_t)\eta_t\mathrm{d}t,\quad V^n(0) = 0.
		\end{equation*}	
		Conversely, by the triangle inequality, we obtain
		\begin{equation*}
			\lim_{\varepsilon\downarrow 0}\frac{1}{\varepsilon}\big(d(\hat\alpha^n, \alpha + \varepsilon\eta) - d(\hat\alpha^n, \alpha) \big) \le \varepsilon\mathbb{E}\big[ \sup_{t \in [0,T]}|\eta_t|^4 \big]^{1/4}.
		\end{equation*}
		Hence, $J^{\varepsilon_n}_n$ is G\^ateaux-differentiable and the optimality of $\hat\alpha^n$ for $J^{\varepsilon_n}_n$ implies that
		\begin{align*}
			0\le \frac{\mathrm{d}}{\mathrm{d}\varepsilon}J^{\varepsilon_n}_n(\hat\alpha^n + \varepsilon \eta)_{|_{\varepsilon = 0}} =& \frac{\mathrm{d}}{\mathrm{d}\varepsilon}J_n(\hat\alpha^n + \varepsilon \eta)_{|_{\varepsilon = 0}} + \lim_{\varepsilon\downarrow 0} (2\varepsilon_n)^{1/2}\frac{1}{\varepsilon}d(\hat\alpha^n,\hat\alpha^n + \varepsilon\eta)  \\
			\leq	&  \mathbb{E}\big[\int_0^T\partial_xf\big(t, X^{n,\hat\alpha^n}_{t}, \hat\alpha^n_{t} \big)V^n_{t} + \partial_{\alpha}f\big( t, X^{n,\hat\alpha^n}_{t}, \hat\alpha^n_{t} \big)\eta_{t}\mathrm{d}t\\
			&+ \partial_xg(X^{n,\hat\alpha^n}_T)V^n_T \big] + \big(2\varepsilon_n)^{1/2}(E[\sup_t|\eta_{t}|^{4}] \big)^{1/4}\\
			\leq	& \mathbb{E}\big[\int_0^T\partial_\alpha H_n\big(t, X^{n,\hat\alpha}_{t}, P^{n,\hat\alpha^n}_{t}, \hat\alpha^n_{t} \big)\eta_{t}\mathrm{d}t \big] + C_M\varepsilon_n^{1/2},
		\end{align*}
		for a constant $C_M>0$. The inequality follows since $\hat\alpha^n\in \mathcal{A}$, and $H_n$ denotes the Hamiltonian of the problem corresponding to the drift $b_n$ given by
		\begin{equation*}
			H_n(t,x,p,a) := f(t, x,a) + \{b_n(t,x,a)+\int_{\R_0}\gamma(z)\nu(\diffns z)\}p, 
		\end{equation*}
		and $(P^{n,\hat\alpha^n}, Z^{n,\hat\alpha^n},R^{n,\hat\alpha^n}(z))$ the adjoint processes satisfying 
		\begin{equation*}
			\mathrm{d}P^{n,\hat\alpha^n}_{t} = -\partial_xH_n(t, X^{n,\hat\alpha^n}_{t}, P^{n,\hat\alpha^n}_{t}, \hat\alpha^n_{t})\mathrm{d}t + Z^{n,\hat\alpha^n}_{t}\mathrm{d}B_{t}+\int_{\R_0}R_t^{n,\hat\alpha^n}(z)\tilde N(\diffns z,\diffns t).
		\end{equation*}
		By standard arguments, we have
		\begin{equation*}
			C_M\varepsilon_n^{1/2} +\partial_\alpha H_n(t, X^{n,\hat\alpha^n}_{t}, P^{n,\hat\alpha^n}_{t}, \hat\alpha^n_{t})\cdot (\beta - \hat\alpha^n_t)) \ge 0 \quad \PP\otimes \mathrm{d}t \text{-a.e.}
		\end{equation*}
		Since $b_{2,n}$ does not depend on $\alpha$, this reduces to
		\begin{equation*}
			C_M\varepsilon_n^{1/2} + \big\{ \partial_{\alpha}f(t, X^{n,\hat\alpha^n}_{t}, \hat\alpha^n_{t}) + \partial_{\alpha}b_{1}\big(t, X^{n,\hat\alpha^n}_{t}, \hat\alpha^n_{t}\big)P^{n,\hat\alpha^n}_t \big\}\cdot(\beta - \hat\alpha^n_t) \ge 0 \quad \PP \otimes \diffns t\text{-a.e.}
		\end{equation*} 
		Taking the limit as $n \rightarrow \infty$, Lemmas \ref{lem:conv.Xnn} and \ref{lem:conv.y.phi} ensure that $X^{n,\hat\alpha^n}_t \to X_t^{\hat\alpha}$ and $P^{n,\hat\alpha^n}_t \to P^{\hat\alpha}_t$ $\PP$-a.s. for all $t\in [0,T]$.
		Thus, using the fact that $\hat\alpha^n\to \alpha$,  we obtain
		\begin{equation*}
			\big\{ \partial_{\alpha}f(t, X^{\hat\alpha}_t, \hat\alpha_t) + \partial_{\alpha}b_{1}\big(t, X^{\hat\alpha}_t,  \hat\alpha_t \big)P^{\hat\alpha}_t \big\}\cdot(\beta - \hat\alpha_t) \ge 0 \quad \PP\otimes \mathrm{d}t\text{-a.e.}
		\end{equation*}
		This shows \eqref{eq:nec.cond}, which concludes the proof.
	\end{proof}

	\begin{proof}[Proof of Theorem \ref{thm:suff}]
		Let $\hata \in \mathcal{A}$ satisfy \eqref{eq:suff.con} and $\alpha'$ an arbitrary element of $\mathcal{A}$.
		We would like to show that $J(\hata) \le J(\alpha')$.
		Let $n \in \mathbb{N}$ be arbitrarily chosen.
		By definition, we have
		\begin{align*}
			&J_n(\hata) - J_n(\alpha')\\
			& = \E\big[g(X^{n,\hata}_T) - g(X^{n,\alpha'}_T) + \int_0^Tf(u, X_u^{n,\hata}, \hata_u) - f(u, X_u^{n,\alpha'}, \alpha'_u)\diff u  \big]	\\
			&\le \E\big[\partial_xg(X^{n,\hata}_T)\big\{X^{\hata}_T -X^{n,\alpha'}_T\big\} + \int_0^T\big\{ b_n(u, X^{n,\alpha'}_u, \alpha'_u) - b_n(u, X_u^{n,\hata},\hata_u)\big\} P^{n,\hata}_u\diff u\\
			&\quad + \int_0^T H_n(u, X_u^{n,\hata}, P_u^{n,\hata}, \hata_u) - H_n(u, X_u^{n,\alpha'},P_u^{n,\hata}, \alpha'_u)\diff u  \big],
		\end{align*}	
		where we used the definition of $H_n$ and the fact that $g$ is concave.
		Since $P_n^{\hata}$ satisfies
		\begin{equation*}
			P^{n,\hata}_t = \E\big[\Phi_{t,T}^{n,\hata} \partial_xg( X^{n,\hata}_T) + \int_t^T\Phi_{t,u}^{n,\hata} \partial_xf(u, X_u^{n,\hata}, \hata_u)\mathrm{d}u\mid \mathcal{F}_t \big],
		\end{equation*}
		it follows by the martingale representation (see \cite{TangLi1994}) and the It\^o formula that there is a square integrable progressive process $(P^{\hata}_n,Z^{\hata}_n, R_u^{n,\hat\alpha})$  such that $P_n^{\hata}$ satisfies 
		\begin{align*}
			P^{n,\hata}_t =& \partial_xg(X^{n,\hata}_T) + \int_t^T\partial_xH_n(u, X^{n,\hata}_u, P_u^{n,\hata},\hata_u)\diff u - \int_t^TZ_u^{n,\hata}\diff W_u\\
			&-\int_t^T\int_{\R_0}R_u^{n,\hat\alpha}(z)\tilde N(\diffns z,\diffns u).
		\end{align*}
		Since $b_n$ is smooth, so is $H_n$.
		Therefore, by It\^o's formula, we have
		\begin{align}\label{eqnewit1}
			&P^{n,\hata}_T\big\{X_T^{n,\hata} - X^{n,\alpha'}_T\big\}\notag\\
			=& \int_0^TP^{n,\hata}_u\big\{b_n(u, X^{n,\hata}_u,\hata_u) - b_n(u, X^{n,\alpha'}_u,\alpha'_u) \big\}\diff u\notag\\
			& - \int_0^T\big\{X^{n,\hata}_u - X^{n,\alpha'}_u \big\}\partial_xH_n(u, X^{n,\hata}_u, P_u^{n,\hata},\hata_u)\diff u \notag\\
			&+ \int_0^T\big\{X^{n,\hata}_u - X^{n,\alpha'}_u \big\} Z^{n,\hata}_u\diff W_u+\int_0^T\int_{\R_0}\big\{X^{n,\hata}_u - X^{n,\alpha'}_u \big\}R_u^{n,\hat\alpha}(z)\tilde N(\diffns z,\diffns t).
		\end{align}
		The stochastic integrals in \eqref{eqnewit1} are local martingales. We now verify that these local martingales are true martingales. Set $\Delta X_u^n:=X_u^{n,\hat\alpha}-X_u^{n,\alpha'}$.
			The state moment estimates and the square-integrability of the
			approximating adjoint yield
			
			$$
			\E\big[\sup_{u\leq T}|\Delta X_u^n|^2\big]<\infty,
			\quad
			\E\big[\int_0^T\big(
			|Z_u^{n,\hat\alpha}|^2+
			\int_{\R_0}|R_u^{n,\hat\alpha}(z)|^2\nu(dz)
			\big)\mathrm{d}u\big]<\infty.
			$$
			Next, define $	M_t^n:=\int_0^t\Delta X_{u-}^nZ_u^{n,\hat\alpha}\mathrm{d}B_u,$ and $		L_t^n:=\int_0^t\int_{\R_0}
			\Delta X_{u-}^n R_u^{n,\hat\alpha}(z)\,
			\widetilde N(\mathrm{d}z,\mathrm{d}u).$
			By the Burkholder-Davis-Gundy (B-D-G) and Cauchy-Schwarz inequalities,
			\begin{align}\label{eqbounMn1}
				\E\big[\sup_{t\leq T}|M_t^n|\big]
				\leq& C\E\big[
				\big(\int_0^T|\Delta X_{u-}^n|^2
				|Z_u^{n,\hat\alpha}|^2\mathrm{d}u\big)^{1/2}
				\big]\notag\\
				\leq& C
				\E\big[\sup_{u\leq T}|\Delta X_u^n|^2\big]^{1/2}
				\E\big[\int_0^T|Z_u^{n,\hat\alpha}|^2\mathrm{d}u\big]^{1/2}
				<\infty.
			\end{align}
			Similarly, using once more the B-D-G and Cauchy-Schwarz inequalities,
			\begin{align}\label{eqbounLn1}
				\E\big[\sup_{t\leq T}|L_t^n|\big]
				\leq &C\E\big[
				\big(\int_0^T\int_{\R_0}
				|\Delta X_{u-}^n|^2|R_u^{n,\hat\alpha}(z)|^2
				N(\mathrm{d}z,\mathrm{d}u)\big)^{1/2}
				\big]\notag\\
				\leq & C
				\E\big[\sup_{u\leq T}|\Delta X_u^n|^2\big]^{1/2}
				\E\big[\int_0^T\int_{\R_0}
				|R_u^{n,\hat\alpha}(z)|^2\nu(\mathrm{d}z)\mathrm{d}u\big]^{1/2}
				<\infty.
			\end{align} 
			Fix $n$ and define the continuous increasing process
			\[
			A_t^n
			:=
			\int_0^t|\Delta X_u^n Z_u^{n,\hat\alpha}|^2\mathrm{d}u
			+\int_0^t\int_{\R_0}
			|\Delta X_{u-}^nR_u^{n,\hat\alpha}(z)|^2\nu(\mathrm{d}z)\mathrm{d}u,
			\]
			and the stopping times $\tau_k:=\inf\{t\in[0,T]:A_t^n\geq k\}\wedge T,
			\,\, \inf\varnothing=\infty.$ Since $A^n$ is continuous, $A_{\tau_k}^n\leq k$.
			The stopped processes $M^{n,\tau_k}$ and $L^{n,\tau_k}$
			are therefore square-integrable martingales, and the
			stochastic-integral isometries give
			\[
			\E\big[|M_{T\wedge\tau_k}^n|^2\big]
			+\E\big[|L_{T\wedge\tau_k}^n|^2\big]
			=\E[A_{\tau_k}^n]\leq k.
			\]
			In particular, $\E[M_{T\wedge\tau_k}^n]=\E[L_{T\wedge\tau_k}^n]=0.$ Moreover,
			\[
			A_T^n
			\leq
			\sup_{u\leq T}|\Delta X_u^n|^2
			\int_0^T\big(
			|Z_u^{n,\hat\alpha}|^2
			+\int_{\R_0}|R_u^{n,\hat\alpha}(z)|^2\nu(\mathrm{d}z)
			\big)\mathrm{d}u
			<\infty,
			\quad\text{a.s.}
			\]
			Since $A_T^n<\infty$ almost surely, $\tau_k=T$ for all
			sufficiently large $k$, almost surely.
			Furthermore, the running suprema of $M^n$ and $L^n$ are
			integrable by \eqref{eqbounMn1} and \eqref{eqbounLn1}.
			Thus, for every $0\leq s\leq t\leq T$, passing to the limit
			in the stopped martingale identities by conditional
			dominated convergence gives $\E[M_t^n\mid\mathcal F_s]=M_s^n,
			\quad\E[L_t^n\mid\mathcal F_s]=L_s^n.$
			Therefore, $M^n$ and $L^n$ are true martingales.
			Since both start at zero, $\E[M_T^n]=\E[L_T^n]=0.$
			Taking expectations on both sides of \eqref{eqnewit1}
			therefore yields

		\begin{align*}
			J_n(\hata) - J_n(\alpha') &\le \E\big[- \int_0^T\big\{X^{n,\hata}_u - X^{n,\alpha'}_u \big\}\partial_xH_n(u, X^{n,\hata}_u, P_u^{n,\hata},\hata_u)\diff u \\
			&\quad + \int_0^T H_n(u, X_u^{n,\hata}, P_u^{n,\hata}, \hata_u) - H_n(u, X_u^{n,\alpha'},P_u^{n,\hata}, \alpha'_u)\diff u   \big]\\
			&\le \E\big[\int_0^T \partial_\alpha H_n(u, X_u^{n,\hata}, P_u^{n,\hata}, \hata_u)\cdot(\hata_u - \alpha'_u)\diff u  \big],
		\end{align*}
		where the latter inequality follows by convexity of $(x,a)\mapsto H(t,x,P^{\hat{\alpha}}_t,a)$. Substituting this into the expression $J(\hata) - J(\alpha')$, we have
		\begin{align*}
			J(\hata) - J(\alpha') & = J(\hata) - J_n(\hata) + J_n(\hata) - J_n(\alpha') + J_n(\alpha') - J(\alpha')\\
			&\le J(\hata) - J_n(\hata) + \E\big[\int_0^T \partial_\alpha H_n(u, X_u^{n,\hata}, P_u^{n,\hata}, \hata_u)\cdot(\hata_u - \alpha'_u)\diff u  \big]\\
			&\quad  + J_n(\alpha') - J(\alpha').
		\end{align*}
		Since $b_{2,n}$ is independent of $\alpha$, we have 
		$$\partial_\alpha H_n(u, X_u^{n,\hata}, P_u^{n,\hata}, \hata_u) = \partial_\alpha b_{1}(u, X^{n,\hata}_u,\hata_u)P^{n,\hata}_u + \partial_\alpha f(u, X^{n,\hata}_u,\hata_u).$$
		Passing to the limit as $n \rightarrow \infty$, it follows from Lemmas \ref{lem:conv.Xnn}, \ref{lem:J.continuous} and \ref{lem:conv.y.phi} that
		\begin{align*}
			J(\hata) - J(\alpha') \le \E\big[\int_0^T \partial_\alpha H(u, X^{\hata}_u, P^{\hata}_u, \hata_u)\cdot(\hata_u - \alpha'_u)\diff u  \big].
		\end{align*}
		By \eqref{eq:suff.con}, the integrand on the right side is nonpositive. Hence $J(\hat\alpha)\leq J(\alpha')$
		for every $\alpha'\in\mathcal A$, proving that
		$\hat\alpha$ is optimal.
	\end{proof}
	
	\begin{proof}[Proof of Lemma \ref{lem:conv.Xnn}]
		The proof of $(i)$ is standard, it follows from the linear growth property of $b_n:=b_1+b_{2,n}+b_3$ uniformly in $n$, i.e. $|b_n(t,x,a)| \le C(1 + |x| + |a|)$ for some $C>0$ and all $n\ge1$.
		
		Let us turn to the proof of $(ii)$. 
		By adding and subtracting the same term and applying the mean value theorem, we obtain
		\begin{align*}
			X^{n,\alpha^1}_t - X^{\alpha^2}_t 
			= &\int_0^t\int_0^1\big\{\partial_xb_{1}(  \Lambda_n(\lambda,s),\alpha^1_s) +\partial_xb_{3}(  \Lambda_n(\lambda,s))\\& + \partial_xb_{2,n}\big(\Lambda_n(\lambda,s)\big)\big\}\mathrm{d}\lambda\times(X^{n,\alpha^1}_s - X^{\alpha^2}_s)\,\mathrm{d}s\\
			& + \int_0^tb_1( X^{\alpha^2}_s,\alpha^1_s) - b_1( X^{\alpha^2}_s, \alpha^2_s))\diff s+\int_0^t\big( b_{2,n}( X^{\alpha^2}_s)- b_{2}( X^{\alpha^2}_s\big)\diffns s,
		\end{align*}
		where $\Lambda_n(\lambda,t)$ is the process given by $\Lambda_n(\lambda,t):= \lambda X^{n,\alpha^1}_t + (1 - \lambda)X^{\alpha^2}_t$. Therefore, we obtain that $X^{n,\alpha^1} - X^{\alpha^2}$ admits the representation
		\begin{align*}
			&X^{n,\alpha^1}_t - X^{\alpha^2}_t\\
			=& \int_0^t\exp\big(\int_{s}^t\int_0^1\partial_xb_{1}(  \Lambda_n(\lambda,r),\alpha^1_r)+\partial_xb_{3}(  \Lambda_n(\lambda,s)) + \partial_xb_{2,n}(\Lambda_n(\lambda,r))\mathrm{d}\lambda\mathrm{d}r \big)\\
			&\times \Big\{b_1( X^{\alpha^2}_s,\alpha^1_s) - b_1( X^{\alpha^2}_s, \alpha^2_s) +\big( b_{2,n}( X^{\alpha^2}_s)- b_{2}( X^{\alpha^2}_s)\big)\big\}\mathrm{d}s.
		\end{align*}
		Hence, taking the expectation on both sides above and then using twice the Cauchy-Schwarz inequality, we have that
		\begin{align}	\label{eq:estim.diff.x}
			\notag
			\E\big[|X^{n,\alpha^1}_t - X^{\alpha^2}_t|^2\big]
			\le&8T^2\E\big[\int_0^t \exp\big\{4\int_{s}^t\int_0^1\big(\partial_xb_{1}(  \Lambda_n(\lambda,r),\alpha^1_r) +\partial_xb_{3}(  \Lambda_n(\lambda,s))\\
			&+ \partial_xb_{2,n}(\Lambda_n(\lambda,r))\big)\mathrm{d}\lambda\mathrm{d}r \big\}\mathrm{d} s\big]^{1/2}\times\E\big[\int_0^{t}|b_1( X^{\alpha^2}_s,\alpha^1_s) \notag\\
			&- b_1( X^{\alpha^2}_s, \alpha^2_s) |^4+ |\big( b_{2,n}( X^{\alpha^2}_s)- b_{2}( X^{\alpha^2}_s)\big)|^4\diff s\big]^{1/2}.
		\end{align}
		By the Lipschitz continuity of $b_1$, we have 
		\begin{align}
			\label{eq:estim.alpha12}
			\E\big[\int_0^T|b_1( X^{\alpha^2}_s,\alpha^1_s) - b_1( X^{\alpha^2}_s, \alpha^2_s) |^4\diff s \big]
			\le& C\E\big[\int_0^T|\alpha^1_s - \alpha^2_s|^4\diff s \big]\notag\\
			\le& C(d(\alpha^1,\alpha^2))^4.
		\end{align}

		Using the fact that the law of $X^{\alpha^2}_t$ has a density with respect to the Lebesgue measure denoted by $\rho^{\alpha^2}_t(x)$, we have
		\begin{align*}
			\E\big[\int_0^{T}| b_{2,n}( X^{\alpha^2}_s)- b_{2}( X^{\alpha^2}_s)|^4\mathrm{d}s\big]=  \int_0^{T}\int_{\mathbb{R}}| b_{2,n}( x)- b_{2}(x)|^4\rho^{\alpha^2}_s(x)\diffns x\mathrm{d}s.
		\end{align*}
		
		Let us turn our attention to the first term in \eqref{eq:estim.diff.x}.
		Since $\Lambda_{n}(\lambda,t)$ takes the form 
		\begin{align}\label{eqtranssde1}
			\Lambda_{n}(\lambda, t) =& x_0 + \int_0^t\big\{\lambda b_{n}(s, X^{n,\alpha^1}_s,\alpha^1_s) + (1-\lambda)b(s, X^{\alpha^2}_s,\alpha^2_s)\big\}\diff s \notag\\
			&+ \sigma B(t)+\int_0^t\int_{\mathbb{R}_0}\gamma (z) N(\mathrm{d}z,\diffns t)\notag\\
			=&  x_0 +\int_0^tb_n^{\lambda,\alpha^2,\alpha^1}(s)\mathrm{d}s + \sigma B_t+\int_0^t\int_{\mathbb{R}_0}\gamma (z) N(\mathrm{d}z,\diffns t),
		\end{align} 
		where $b_n^{\lambda,\alpha^2,\alpha^1}(s):=\lambda b_{n}(s, X^{n,\alpha^1}_s,\alpha^1_s) + (1-\lambda)b(s, X^{\alpha^2}_s,\alpha^2_s).$ We use the Jensen inequality, and Lipschitz continuity of $b_1$ and $b_3$ to get
		\begin{align}
			&\text{ }	\E\big[\exp\big\{4\int_{s}^t\int_0^1\big(\partial_xb_{1}(  \Lambda_n(\lambda,r),\alpha^1_r) +\partial_xb_{3}(  \Lambda_n(\lambda,r))+ \partial_xb_{2,n}(\Lambda_n(\lambda,r))\big)\diff\lambda\mathrm{d}r\big\} \big]	\notag\\
			\le& C\int_0^1 \E\big[\exp\big(8\int_{s}^t \partial_xb_{2,n}(\Lambda_n(\lambda,r))\mathrm{d}r \big) \big]^{1/2}\diff\lambda.
			\label{eq:estime.bprime} 
		\end{align}
		Set $
		\tilde b_{2,n}(x)=\int_0^x  b_{2,n}(y)\diffns y.
		$ Using It\^o's formula for L\'evy processes, one has
		\begin{align*}
			&\tilde b_{2,n}(\Lambda_n(\lambda,t))- \tilde b_{2,n}(\Lambda_n(\lambda,s))\\
			=&	\int_s^tb_{2,n}(\Lambda_n(\lambda,r))b^{\lambda,\alpha^2,\alpha^1}_n(r) \mathrm{d}r+\int_s^t b_{2,n}(\Lambda_n(\lambda,r)) \mathrm{d}B_r+\frac{1}{2}\int_s^t\partial_x  b_{2,n}(\Lambda_n(\lambda,r)) \mathrm{d}r	\notag\\
			&+\int_s^t\int_{\mathbb{R}_0}\big\{ \tilde b_{2,n}(\Lambda_n(\lambda,r)+\gamma(z))-\tilde b_{2,n}(\Lambda_n(\lambda,r))\big\} N(\mathrm{d}z,\mathrm{d}r).
		\end{align*}
		Using the uniform boundedness of $b_{2,n}$ and the linear growth property of  $b^{\lambda,\alpha^2,\alpha^1}_n$, it follows, by an argument similar to those in the proof of Theorem \ref{thmmainresder1}, that
		\begin{align}
			\label{eq:bound.bprime}
			\sup_n\E\big[\exp\big(16\int_{s}^t \partial_x b_{2,n}(\Lambda_n(\lambda,r))\mathrm{d}r \big) \big]^{1/4}
			\leq C.
		\end{align}
		Therefore, putting together \eqref{eq:estim.diff.x}, \eqref{eq:estim.alpha12}, \eqref{eq:estime.bprime} and \eqref{eq:bound.bprime} concludes the proof.
		
		Since $b_k$ is Lipschitz continuous the convergence (iii) follows by classical arguments, the proof is omitted.
	\end{proof}
	
	\begin{proof}[Proof of Lemma \ref{lem:J.continuous}]
		(i)	The continuity of $J_k$ easily follows by the continuity of $f$ and $g$, the dominated convergence and Lemma \ref{lem:conv.Xnn}.

		(ii) is also a consequence of Lemma \ref{lem:conv.Xnn}.
		Indeed, by the linear growth of $\partial_xf$ and $\partial_xg$ and the Cauchy-Schwarz inequality and Fubini's theorem, we have
		\begin{align*}
			| J_n(\alpha) - J(\alpha) |&\le  \E\Big[|g(X^{n,\alpha}_T) - g(X^{\alpha}_T)| + \int_0^T|f(t, X^{n,\alpha}_t, \alpha_t) - f(t, X^{\alpha}_t, \alpha_t) |\diff t \Big]\\
			&\le \E\Big[\int_0^1|\partial_xg(\lambda X^{n,\alpha}_T) + (1-\lambda)X^{\alpha}_T|\diff \lambda|X^{n,\alpha}_T - X^{\alpha}_T|\Big]\\
			&\quad + \E\Big[\int_0^T\int_0^1|\partial_{x}f\big(t, \lambda X^{n,\alpha}_t+(1-\lambda)X^{\alpha}_t,\alpha_t \big)|\diff \lambda|X^{n,\alpha}_t - X^{\alpha}_t|\diff t \Big]\\
			&\le C\E\Big[1+\sup_{t\in [0,T]}\big( |X^{n,\alpha}_t|^2 + |X^\alpha_t|^2\big) \Big]^{1/2}\Big(\sup_{t \in [0,T]}\E\Big[|X^{n,\alpha}_t - X^\alpha_t|^2\Big]\Big)^{1/2}.
		\end{align*}
		Therefore, by Lemma \ref{lem:conv.Xnn} (i) and (ii), we have
		\begin{align*}
			| J_n(\alpha) - J(\alpha) | &\le C\sup_{t\in [0,T]}\E[|X^{n,\alpha}_t - X^{\alpha}_t|^2]^{1/2}
			\le C\varepsilon_n.
		\end{align*}
	\end{proof}
	
	\begin{proof}[Proof of Lemma \ref{lem:conv.y.phi}]
		(i)	The existence of the process $\Phi^{\alpha^n}_n$ is standard, it follows for instance by \cite[Theorem 3.3.2]{KunitaB19}.
		The existence of the first variation $\Phi^{\alpha}$ follows by Theorem \ref{thmmainresder1}.
		We know from \eqref{eqfirstvarX1} that $\Phi^\alpha$ admits the following representation 
		\begin{align*} 
			\Phi^\alpha_{t,s}  =&\exp\big\{\int_t^s\partial_xb_1(X^{\alpha}_u,\alpha_u) \mathrm{d}u+\int_t^s\partial_xb_3(X^{\alpha}_u) \mathrm{d}u+2\big(\tilde b_2(X^{\alpha}_s)- \tilde b_2(X^{\alpha}_t)\notag\\
			&-	\int_t^sb_2^2(X^{\alpha}_u) \mathrm{d}u\-\int_t^sb_2(X^{\alpha}_u) \mathrm{d}B_u\\
			&	-\int_t^s\int_{\mathbb{R}_0}\big\{ \tilde b_2(X^{\alpha}_{u-}+\gamma(z))-\tilde b_2(X^{\alpha}_{u-})\big\} N(\mathrm{d}z,\mathrm{d}u)\big)\big\},
		\end{align*}
		and $\Phi_n^{\alpha^n}$ admits the same representation with $(b_2, X^{\alpha},\alpha)$ replaced by \\$(b_{2,n}, X^{\alpha^n}, \alpha^n)$. The proof of the $L^2$-convergence follows as in the proof of \textbf{Claim 1} (see proof of Theorem \ref{thmmainresder1} in Appendix \ref{appen}).

		\vspace{.2cm}
		(ii)	For this proof, first compute the difference $P_t^{n,\alpha^n}-P^\alpha_t$, add and subtract the terms $\Phi^{\alpha}_{t,T} \partial_xg(X_T^{n,\alpha^n})$ and $\int_t^T\Phi^{\alpha}_{t,u} \partial_xf(u,X_u^{n,\alpha^n}, \alpha^n_u)\diff u$ and then apply H\"older's inequality to obtain
		\begin{align}\label{eq:conv.Y}
			\notag
			&\E[|P_t^{n,\alpha^n}-P^\alpha_t|]\\\notag
			\leq & C_T\big\{\E\big[\big|\Phi^{\alpha}_{t,T}\big|^2\big]^{\frac{1}{2}}\E\big[|\partial_xg( X_T^{n,\alpha^n}) - \partial_xg( X^{\alpha}_T)|^2\big]^{\frac{1}{2}}\\ 
			&+\E\big[|\partial_xg( X_T^{n,\alpha^n})|^2\big]^{\frac{1}{2}}\E\big[\big|\Phi^{n,\alpha^n}_{t,T} - \Phi^{\alpha}_{t,T}\big|^2\big]^{\frac{1}{2}}\notag\\\notag
			&+\E\big[\int_t^T|\Phi^{\alpha}_{t,u}|^2\diff u\big]^{\frac{1}{2}} \E\big[\int_0^T|\partial_xf(u, X^{\alpha}_u, \alpha_u)-\partial_xf(u, X_u^{n,\alpha^n}, \alpha^n_u)|^2\diff u\big]^{\frac{1}{2}}\notag\\
			&+\E\big[\int_0^T|\partial_xf(u, X_u^{n,\alpha^n}, \alpha_u^n)|^2\diff u\big]^{\frac{1}{2}}\E\big[\int_0^T|\Phi_u^{n,\alpha^n}-\Phi^{\alpha}_u|^2\diff u\big]^{\frac{1}{2}}\big\},
		\end{align}
	for some constant $C_T$ depending only on $T$.
			Since $d(\alpha^n,\alpha)\to0$, Lemma \ref{lem:conv.Xnn}(ii)
			implies that $X_u^{n,\alpha^n}\to X_u^\alpha$ in
			$L^2(\Omega)$ for every $u\in[0,T]$.
			Moreover, convergence in $d$ ensures uniform integrability
			of $\{\sup_{u\in[0,T]}|\alpha_u^n|^4\}_{n\ge1}$.
			Together with the uniform state moment bounds in
			Lemma \ref{lem:conv.Xnn}(i), the continuity and quadratic
			growth of $\partial_xf$ and $\partial_xg$ therefore imply
			$\E\big[|\partial_xg(X_T^{n,\alpha^n})-\partial_xg(X_T^\alpha)|^2
			\big]\longrightarrow0 $
			and
			$
			\E\big[\int_0^T
			\big|
			\partial_xf(u,X_u^{n,\alpha^n},\alpha_u^n)
			-\partial_xf(u,X_u^\alpha,\alpha_u)
			\big|^2\,\diff u\big]\longrightarrow0.$
			The moment bounds for $\Phi^\alpha$ established in
			Theorem \ref{thmmainresder1} then show that the first and
			third terms in \eqref{eq:conv.Y} converge to zero.
			Furthermore, the quadratic growth assumptions and the
			uniform fourth-moment bounds for the states and controls
			ensure that the factors involving $\partial_xg$ and
			$\partial_xf$ in the second and fourth terms remain bounded.
			The $L^2$-convergence established in part (i) thus implies
			that these terms also converge to zero.
			This completes the proof of (ii).
	\end{proof}

	\section{Application}\label{secappl}

	In this section, we apply the stochastic maximum principle developed in the previous sections to study the problem introduced in Subsection \ref{sectappsub1}. More precisely the problem is the following:
	
	\begin{prob}\label{probex1}
		Find $\hat \alpha \in \mathcal{A}$ such that 
		\begin{equation}\label{pracex21}
			J(\hat \alpha(\cdot)) = \inf_{\alpha \in \mathcal{A}}\mathbb{E}\big[(X^{\alpha}_T)^2\big],
		\end{equation} 
		where the state process \(X^\alpha\) satisfies
		\begin{equation}\label{pracex11}
			\diffns X^{\alpha}_t 
			=-\big\{\delta_t X^{\alpha}_t + b_t + \alpha_t - \beta \, \mathrm{sgn}(X^{\alpha}_t) \mathbf{1}_{\{|X^{\alpha}_t| > H\}}\big\}\diffns t -\sigma_t \diffns B_t -\int_{\mathbb{R}_0}z N(\mathrm{d}z,\mathrm{d}t),
		\end{equation}
		and the control \(\alpha\) takes values in $\mathbb{A} = [-a,a]$, where $a>0$. 
	\end{prob}
	In order to have an explicit solution, we use the diffusion approximation of surplus process (see for example \cite{ Gran91, norberg99, Schmid08}) so that the SDE \eqref{pracex11} becomes
	\begin{equation*}
		\diffns X^{\alpha}_t 
		=-\big\{\delta_t X^{\alpha}_t + b_t + \alpha_t - \beta \, \mathrm{sgn}(X^{\alpha}_t) \mathbf{1}_{\{|X^{\alpha}_t| > H\}}\big\}\diffns t - \sigma_t \diffns B_t -\big\{\int_{\mathbb{R}_0}z^2\nu(\mathrm{d}z)\big\}\mathrm{d}\bar B_t,
	\end{equation*}
	where $\bar B$ is a Brownian motion independent of $B$.  Note that without loss of generality, we suppose that $\sigma_t=\sigma,\,\delta_t=\delta$ and $b_t=b=0$. The above equation can be written as 
	\begin{equation}\label{pracex112}
		\diffns X^{\alpha}_t 
		=-\big\{\delta X^{\alpha}_t + \alpha_t - \beta\, \mathrm{sgn}(X^{\alpha}_t) \mathbf{1}_{\{|X^{\alpha}_t| > H\}}\big\}\diffns t - \sigma  \diffns B_t - \bar \sigma\mathrm{d}\bar B_t,\,\,t\in[0,T],\,\,\, X^{\alpha}_t=x
	\end{equation}
	where $\bar \sigma=\int_{\mathbb{R}_0}z^2\nu(\mathrm{d}z)$. The main result of this section is the following
	
	\begin{prop}
		Consider the control problem obtained from
		Problem~\ref{probex1} by replacing the jump-diffusion state
		equation with its diffusion approximation
		\eqref{pracex112}. Let $\hat X$ be the unique strong solution
		of the corresponding closed-loop equation. Then the feedback
		control
		\begin{equation}
			\label{eq:max}
			\hat\alpha_t=a\sgn(\hat X_t)
		\end{equation}
		is extremal in the sense that it satisfies the Hamiltonian minimisation condition. Moreover, it is optimal.
	\end{prop}
	\begin{proof}

		Observe that the terminal condition $g:x \mapsto g(x)=x^2$ is even, continuously differentiable, and increasing for $x>0$. Thus the optimal control $\hat{\alpha}$ should minimise the absolute value $|X_T^{\alpha}|$.

		Moreover, $b_2:x \mapsto b_2(x)=\beta \sgn(x)\mathbf{1}_{\{|x|>H\}}$ is odd, bounded and piecewise Lipschitz, whereas $b_3:x \mapsto b_3(x)=-\delta x$ is Lipschitz and grows linearly. The proof follows the same argument as in \cite{MenTan22} (see also \cite{BDM07}) by using uniqueness in law. Applying It\^o–Tanaka's formula to $|X_t^\alpha|$ yields
		$$
		\diff |X_t^{\alpha}| = -\sgn(X_t^{\alpha}) \big\{ \big( \delta X^{\alpha}_t + \alpha_t - \beta\, \mathrm{sgn}(X^{\alpha}_t) \mathbf{1}_{\{|X^{\alpha}_t| > H\}} \big)\mathrm{d}t
		+ \sigma\,\mathrm{d}B_t +\bar \sigma\,\mathrm{d}\bar B_t\big\} +\mathrm{d}L_t^{X^\alpha}(0),
		$$
		where $L_t^{X}(0)$ denotes the local time of $X$ at zero.  
		Thus to minimise $|X_t^{^\alpha}|$, we choose
		$$
		\hat{\alpha}_t = a\sgn\big(\hat{X}_t\big),
		$$
		where $\hat{X}_t=X^{\hat{\alpha}}_t$ is the unique strong solution of the corresponding closed-loop state equation. 
		Set $\hat{P}_t=P^{\hat{\alpha}}_t$. From Theorem~\ref{thm:necc}, the adjoint process satisfies
		$$
		\hat{P}_t = 2\,\mathbb{E}\big[\Phi^{\hat{\alpha}}_{t,T}\,\hat X_T|\mathcal{F}_t\big],
		$$
		where $\Phi^{\hat\alpha}$ is the first variation process associated with the admissible control $\hat\alpha$ satisfies
		\begin{align}\label{abc12}
			\Phi^{\hat{\alpha}}_{t,T} 
			=& \exp\big\{
			\int_t^T \delta\,\mathrm{d}s
			+ \int_{\mathbb{R}}\big(a\sgn(z)-\beta\mathbf{1}_{\{|z|>H\}}\sgn(z)\big)\big(L_T^{\hat{X}}(\mathrm{d}z)-L_t^{\hat{X}}(\mathrm{d}z)\big)
			\big\} \notag \\
			=&\exp\big\{
			\delta(T-t)
			+ 2\big( 
			\bar{b}_2(\hat{X}_T) - \bar{b}_2(\hat{X}_t) \notag \\
			&	+ \delta\int_t^T \big(a\sgn(\hat X_s)-\beta\mathbf{1}_{\{|\hat X_s|>H\}}\sgn(\hat X_s)\big)\hat X_s\,\mathrm{d}s \notag \\
			&
			-\int_t^T\big(a\sgn(\hat X_s)-\beta\mathbf{1}_{\{|\hat X_s|>H\}}\sgn(\hat X_s)\big)^2\,\mathrm{d}s\notag\\
			&	+ \int_t^T \big(a\sgn(\hat X_s)-\beta\mathbf{1}_{\{|\hat X_s|>H\}}\sgn(\hat X_s)\big)\sigma\,\mathrm{d}B_s\notag\\
			&	+ \int_t^T \big(a\sgn(\hat X_s)-\beta\mathbf{1}_{\{|\hat X_s|>H\}}\sgn(\hat X_s)\big)\bar \sigma\,\mathrm{d}\bar B_s
			\big)
			\big\}.
		\end{align}
		
		To obtain this expression, define
		$$
		\bar{b}_2(x) := \int_{0}^{x}\big(a\sgn(z)-\beta\mathbf{1}_{\{|z|>H\}}\sgn(z)\big)\,\mathrm{d}z
		= \int_{0}^{x} b_2(z)\,\mathrm{d}z.
		$$
		The function $\bar{b}_2$ is even. By the Bouleau–Yor formula (see, e.g., \cite[Theorem 78, p.~232]{Prot05}), we have
		$$
		\bar{b}_2(\hat{X}_T)
		= \bar{b}_2(\hat{X}_t)
		+ \int_t^T b_2(\hat{X}_s)\,\mathrm{d}\hat{X}_s
		- \frac{1}{2}\int_{\mathbb{R}} b_2(z)\big(L_T^{\hat{X}}(\mathrm{d}z)-L_t^{\hat{X}}(\mathrm{d}z)\big),
		$$
		which leads directly to the equality in \eqref{abc12}. 
		We now verify that the candidate control satisfies the Hamiltonian
		minimisation condition
		\[
		\hat\alpha_t\in
		\operatorname*{arg\,min}_{\alpha\in[-a,a]}
		H(t,\hat X_t,\alpha,\hat P_t)
		\qquad
		\mathbb P\otimes\mathrm{d}t\text{-a.e.},
		\]
		where $	H(t,x,\alpha,p)	=\big(-\delta x-\alpha
		+\beta\sgn(x)\mathbf{1}_{\{|x|>H\}}\big)p.
		$ For fixed $(t,x,p)$, the Hamiltonian is minimised over $[-a,a]$
		by $\alpha=a\sgn(p)$. Since
		$\hat\alpha_t=a\sgn(\hat X_t)$, it therefore suffices to show that
		$\sgn(\hat P_t)=\sgn(\hat X_t)$.
		Substituting the candidate control into the state equation yields
		$$
		\mathrm{d}\hat{X}_t
		= \big(\delta \hat X_t+a\sgn(\hat X_t)-\beta\mathbf{1}_{\{|\hat X_t|>H\}}\sgn(\hat X_t) \big)\mathrm{d}t
		- \sigma\,\mathrm{d}B_t	- \bar\sigma\,\mathrm{d} \bar B_t, 
		\quad \hat{X}_0 = x.
		$$
		The drift of this SDE is an odd function. Define $\tilde{B}_t := -B_t$ and $\tilde{\bar B}_t := -\bar B_t$, which are Brownian motion with the same laws as $B$ and $\bar B$, respectively. Then the process $(-\hat{X}_t)$ satisfies
		$$
		\mathrm{d}(-\hat{X}_t)
		= \big(-\delta \hat X_t+a\sgn(-\hat X_t)-\beta\mathbf{1}_{\{|\hat X_t|>H\}}\sgn(-\hat X_t) \big)\mathrm{d}t
		-\sigma\,\mathrm{d}(-B_t)	-\bar\sigma\,\mathrm{d} (-\bar B_t).
		$$

		By the strong Markov property and weak uniqueness,
		conditionally on $\mathcal F_\tau$, the processes
			$(-\hat X_s,\tilde B_s-\tilde B_\tau,
			\tilde{\bar B}_s-\tilde{\bar B}_\tau)_{\tau\le s\le T}$
			and
			$(\hat X_s,B_s-B_\tau,\bar B_s-\bar B_\tau)_{\tau\le s\le T}$ have the same law on $\{\tau<T\}$, where
		$\tau:=\inf\{s\ge t:\hat X_s=0\}$.
		In particular, $-\hat X$ and $\hat X$ have the same
		conditional distribution on $[\tau,T]$ on $\{\tau<T\}$.

		Set $S:=\tau\wedge T$ and define the stopped-path sigma-algebra $
			\mathcal G_t:=\sigma\bigl(\hat X_s:t\leq s\leq S\bigr),
			$ understood as the sigma-algebra generated by $(\hat X_{s\wedge S})_{t\leq s\leq T}$. It is clear that $\sigma(\hat X_t)\subseteq\mathcal G_t\subseteq\mathcal F_S$,
			and $\Phi_{t,S}^{\hat\alpha}$ is $\mathcal G_t$-measurable.
			Let $v(s,x):=\E_{s,x}\big[\Phi_{s,T}^{\hat\alpha}\hat X_T\big],
			\,\,\, v(T,x)=x,$ where $\E_{s,x}$ denotes expectation for the dynamics started
			from $x$ at time $s$. By the tower property, the strong Markov property, and since $(S,\hat X_S)$ is $\mathcal G_t$-measurable, we have
			$$\E\big[
			\Phi_{S,T}^{\hat\alpha}\hat X_T\,\big|\,\mathcal G_t
			\big]=\E\big[\E\big[
			\Phi_{S,T}^{\hat\alpha}\hat X_T\,\big|\,\,\mathcal F_S
			\big]\big|\mathcal G_t
			\big]=\E\big[v(S,\hat X_S)\big|\mathcal G_t
			\big]=v(S,\hat X_S).$$
			Using the Markov property at time $t$, the multiplicative property
			of $\Phi^{\hat\alpha}$, and the tower property, we obtain
			\begin{align}\label{eqmark1}
				\hat P_t
				=&2\E\big[
				\Phi_{t,T}^{\hat\alpha}\hat X_T
				\big|\mathcal F_t
				\big]
				=2\E\big[
				\Phi_{t,T}^{\hat\alpha}\hat X_T
				\big|\hat X_t
				\big]\notag\\
				=&2\E\big[
				\Phi_{t,S}^{\hat\alpha}
				\E\big[
				\Phi_{S,T}^{\hat\alpha}\hat X_T
				\big|\mathcal G_t
				\big]
				\big|\hat X_t
				\big]=2\E\big[
				\Phi_{t,S}^{\hat\alpha}v(S,\hat X_S)
				\big|\hat X_t
				\big]\notag\\
				=&2\E\big[
				\Phi_{t,\tau}^{\hat\alpha}v(\tau,0)
				\mathbbm{1}_{\{\tau\leq T\}}
				\,\big|\,\hat X_t
				\big]+
				2\E\big[
				\Phi_{t,T}^{\hat\alpha}\hat X_T
				\mathbbm{1}_{\{\tau>T\}}
				\big|\hat X_t
				\big].
			\end{align}
			Here we used $\hat X_\tau=0$ on the event $\{\tau\leq T\}$ and
			$v(T,x)=x$. We next show that $v(s,0)=0$ for every $s\in[t,T]$. Fix $s\in[t,T]$
			and consider the dynamics started from zero at time $s$. By the same reflection argument as above, and since $\bar{b}_2$, $\big(a\sgn(x)-\beta\mathbf{1}_{\{|x|>H\}}\sgn(x)\big)^2$, and $\big(a\sgn(x)-\beta\mathbf{1}_{\{|x|>H\}}\sgn(x)\big)x$, are even functions, while $a\sgn(z)-\beta\mathbf{1}_{\{|z|>H\}}\sgn(z)$ is odd, we have
			\begin{align*}
				v(s,0) 
				=& \mathbb{E}_{s,0}\big[ 
				\exp\big\{
				\delta (T-s)
				+ 2\big( 
				\bar{b}_2(-\hat{X}_T) - \bar{b}_2(0) \notag \\
				&	+ \delta\int_s^T \big(a\sgn(-\hat X_r)-\beta\mathbf{1}_{\{|\hat X_r|>H\}}\sgn(-\hat X_r)\big)(-\hat X_r)\,\mathrm{d}r \notag \\
				&
				-\int_s^T\big(a\sgn(-\hat X_r)-\beta\mathbf{1}_{\{|\hat X_r|>H\}}\sgn(-\hat X_r)\big)^2\,\mathrm{d}r\notag\\
				&	+ \int_s^T \big(a\sgn(-\hat X_r)-\beta\mathbf{1}_{\{|\hat X_r|>H\}}\sgn(-\hat X_r)\big)\sigma\,\mathrm{d}(-B_r)\notag\\
				&	+ \int_s^T \big(a\sgn(-\hat X_r)-\beta\mathbf{1}_{\{|\hat X_r|>H\}}\sgn(-\hat X_r)\big)\bar \sigma\,\mathrm{d}(-\bar B_r)
				\big)
				\big\}\times (-\hat{X}_T)  \big] \\
				=& -\mathbb{E}_{s,0}\big[ 
				\Phi^{\hat{\alpha}}_{s,T} \hat{X}_T \big].
			\end{align*}
			By reflection symmetry of the dynamics started from zero,
			and since the exponential weight is invariant under simultaneous
			reflection of the state and both driving Brownian motions, we obtain $v(s,0)= -v(s,0)$.
			Hence $v(s,0)=0$. Substituting into~\eqref{eqmark1} yields
			\[
			\hat P_t
			=\E\big[
			\Phi_{t,T}^{\hat\alpha}\hat X_T
			\mathbbm{1}_{\{\tau>T\}}
			\big|\hat X_t
			\big].
			\]
			Since $\hat X$ has continuous path, we have $\sgn(\hat{X}_t) = \sgn(\hat{X}_s)$ for all $t \le s \le T$ on the event  $\{\tau > T\}$. Thus, the term inside the expectation either vanishes or has the opposite sign as $\hat{X}_t$. Consequently, $\sgn(\hat P_t) = \sgn(\hat{X}_t)$, and $\hat{\alpha}_t$ is an extremal point. 
		
		As for the optimality, set $c:=(\sigma^2+\bar\sigma^2)^{1/2}>0$ and consider once more the candidate feedback control
		\[
		\hat\alpha_t:=a\sgn(\hat X_t),
		\]
		where $\hat X$ denotes the unique strong solution of the
		corresponding closed-loop state equation. 
		Fix an arbitrary $\alpha\in\mathcal A$ and set
		$R_t^\alpha:=|X_t^\alpha|$. By Tanaka's formula,
		\begin{align*}
			\mathrm dR_t^\alpha
			={}&\big(
			-\delta R_t^\alpha
			+\beta\mathbf1_{\{R_t^\alpha>H\}}
			-\alpha_t\sgn(X_t^\alpha)
			\big)\mathrm dt\\
			&-\sigma\sgn(X_t^\alpha)\,\mathrm dB_t
			-\bar\sigma\sgn(X_t^\alpha)\,\mathrm d\bar B_t
			+\mathrm dK_t^\alpha,
		\end{align*}
		where $K_t^\alpha:=L_t^{X^\alpha}(0)$ is continuous,
		nondecreasing, starts from zero, and satisfies $
		\int_0^T\mathbf1_{\{R_t^\alpha>0\}}\,\mathrm dK_t^\alpha=0.$
		Here the local time is normalised according to Tanaka's formula. Since the quadratic variation of $X^\alpha$ is $c^2t$,
		the occupation time formula(see \cite[p. 216]{Prot05}) yields $\int_0^T\mathbf1_{\{X_t^\alpha=0\}}\,\mathrm dt=0,\,\,\text{a.s.}$ Hence, the continuous local martingale
		\[
		W_t^\alpha
		:=-\frac1c\big(
		\sigma\int_0^t\sgn(X_s^\alpha)\,\mathrm dB_s
		+\bar\sigma\int_0^t\sgn(X_s^\alpha)\,\mathrm d\bar B_s
		\big)
		\]
		has quadratic variation $t$ and is therefore a Brownian motion.
		Define
		\[
		f(t,r):=-\delta r-a+\beta\mathbf1_{\{r>H\}},
		\quad
		q_t^\alpha:=a-\alpha_t\sgn(X_t^\alpha)\geq0.
		\]
		Then
		\begin{equation}\label{eq:reflected-arbitrary}
			\mathrm dR_t^\alpha
			=\bigl(f(t,R_t^\alpha)+q_t^\alpha\bigr)\mathrm dt
			+c\,\mathrm dW_t^\alpha+\mathrm dK_t^\alpha.
		\end{equation}
		
		On the same probability space, let
		$(\underline R,\underline K)$ be the unique strong solution
		of the reflected SDE
		\begin{equation}\label{eq:reflected-minimal}
			\mathrm d\underline R_t
			=f(t,\underline R_t)\,\mathrm dt
			+c\,\mathrm dW_t^\alpha+\mathrm d\underline K_t,
			\qquad
			\underline R_0=|x|,
		\end{equation}
		where $\underline R\geq0$ and $\underline K$ increases only
		when $\underline R=0$. Existence and uniqueness follow from
		the well-posedness result for one-dimensional reflected SDEs
		with measurable drift of linear growth (see for example \cite{ZHANG1994}). We next prove that
		\[
		\underline R_t\leq R_t^\alpha
		\qquad\text{for all }t\in[0,T],\quad\text{a.s.}
		\]

		For $r\ge 0$, define
		\[
		F(r):=\int_0^r
		\exp\big(-\frac{2\beta}{c^2}(y-H)^+\big)\mathrm dy=
		\begin{cases}
			r, & 0\leq r\leq H,\\[1mm]
			H+\lambda^{-1}\bigl(1-e^{-\lambda(r-H)}\bigr),
			& r>H,
		\end{cases}
		\]
		where $\lambda:=2\beta/c^2>0$. Thus $F$ is strictly increasing, bounded, and globally
		Lipschitz, with
		\[
		F'(r)=e^{-\lambda(r-H)^+},
		\qquad 0<F'(r)\leq1.
		\]
		Moreover, $F\in C^1([0,\infty))$ has a globally Lipschitz
		derivative, and its weak second derivative is given by
		\[
		F''(r)=-\lambda\mathbf1_{\{r>H\}}F'(r),
		\quad\text{for a.e. }r\geq0.
		\]
		
		Set $\ell:=H+\lambda^{-1}$. The function $F$ maps
		$[0,\infty)$ bijectively onto $[0,\ell)$, with inverse
		\[
		F^{-1}(u)=
		\begin{cases}
			u, & 0\leq u\leq H,\\[1mm]
			H-\lambda^{-1}\log\bigl(1-\lambda(u-H)\bigr),
			& H<u<\ell.
		\end{cases}
		\]
		For every finite $m$, $F'$ is bounded away from zero on
		$[0,m]$. Consequently, $F^{-1}$ is Lipschitz on $[0,F(m)]$. Define, for $u\in[0,\ell)$,
		\[
		A(t,u):=
		F'\bigl(F^{-1}(u)\bigr)
		\bigl(-\delta_tF^{-1}(u)-a\bigr),
		\qquad
		C(u):=cF'\bigl(F^{-1}(u)\bigr).
		\]
		Since $F'$, $F^{-1}$, and the identity function are
		Lipschitz and bounded on the relevant compact intervals,
		and $\delta$ is bounded, $A(t,\cdot)$ is locally Lipschitz,
		uniformly in $t\in[0,T]$. In addition,
		\[
		C(u)=
		\begin{cases}
			c, & 0\leq u\leq H,\\
			c\bigl(1-\lambda(u-H)\bigr), & H<u<\ell,
		\end{cases}
		\]
		so $C$ is globally Lipschitz on $[0,\ell)$. Setting $U_t:=F(\underline R_t)$ and
		$V_t:=F(R_t^\alpha)$, the generalised It\^o formula and the fact that $
		\frac{c^2}{2}F''(r)
		=-\beta\mathbf1_{\{r>H\}}F'(r)
		\,\,\text{a.e.},
		$ yield
		\begin{align*}
			\mathrm dU_t
			&=A(t,U_t)\,\mathrm dt+C(U_t)\,\mathrm dW_t^\alpha
			+\mathrm d\underline K_t,\\
			\mathrm dV_t
			&=\left[
			A(t,V_t)
			+F'\bigl(F^{-1}(V_t)\bigr)q_t^\alpha
			\right]\mathrm dt
			+C(V_t)\,\mathrm dW_t^\alpha
			+\mathrm dK_t^\alpha.
		\end{align*}
		Here the reflection terms are unchanged because $F'(0)=1$.
		For integers $m>|x|$, set
		\[
		\tau_m
		:=\inf\{t\geq0:\underline R_t\vee R_t^\alpha\geq m\}
		\wedge T,
		\qquad D_t:=U_t-V_t.
		\]
		Applying It\^o's formula to $(D_{t\wedge\tau_m}^+)^2$ yields
		\begin{align*}
			(D_{t\wedge\tau_m}^+)^2
			={}&2\int_0^{t\wedge\tau_m}
			D_s^+\bigl(A(s,U_s)-A(s,V_s)\bigr)\,\mathrm ds-2\int_0^{t\wedge\tau_m}
			D_s^+F'(R_s^\alpha)q_s^\alpha\,\mathrm ds\\
			&+\int_0^{t\wedge\tau_m}
			\mathbf1_{\{D_s>0\}}
			|C(U_s)-C(V_s)|^2\,\mathrm ds\\
			&+2\int_0^{t\wedge\tau_m}
			D_s^+\bigl(C(U_s)-C(V_s)\bigr)\,\mathrm dW_s^\alpha+2\int_0^{t\wedge\tau_m}
			D_s^+\,\mathrm d(\underline K_s-K_s^\alpha).
		\end{align*}
		The second term on the right side is nonpositive. Furthermore,
		\[
		D_s^+\,\mathrm d\underline K_s=0,
		\qquad
		-D_s^+\,\mathrm dK_s^\alpha\leq0.
		\]
		Indeed, $\underline K$ increases only when $U=0$, and
		$V\geq0$. The stopped stochastic integral is a
		square-integrable martingale, since its integrand is bounded.
		Taking expectations and using the local Lipschitz bounds,
		we obtain
		\[
		\E\bigl[(D_{t\wedge\tau_m}^+)^2\bigr]
		\leq C_m\int_0^t
		\E\bigl[(D_{s\wedge\tau_m}^+)^2\bigr]\,\mathrm ds.
		\]
		Since $D_0=0$, Gronwall's inequality implies
		$D_{t\wedge\tau_m}^+=0$ almost surely. Since $\underline R$ and $R^\alpha$ have continuous,
		nonexplosive paths, $\tau_m=T$ for all sufficiently large
		$m$, almost surely. Consequently, the stopped comparison
		extends to the whole interval $[0,T]$, giving
		$U_t\leq V_t$ simultaneously for all $t\in[0,T]$.
		Since $F^{-1}$ is strictly increasing, we have
		\[
		\underline R_t\leq R_t^\alpha
		\qquad\text{for all }t\in[0,T],\quad\text{a.s.}
		\]
		
		Under the candidate feedback, Tanaka's formula shows that
		$|\hat X|$ solves the reflected equation
		\eqref{eq:reflected-minimal}, driven by its associated
		Brownian motion. By uniqueness in law, $
		\mathcal L(\underline R)=\mathcal L(|\hat X|).$	Therefore,
		\[
		J(\hat\alpha)
		=\E[|\hat X_T|^2]
		=\E[\underline R_T^2]
		\leq\E[(R_T^\alpha)^2]
		=J(\alpha).
		\]
		Since $\alpha\in\mathcal A$ was arbitrary, $\hat\alpha$
		is optimal.
	\end{proof}
	\subsection{Numerical Results}
	In this section, we present a numerical illustration of the impact of jump characteristics on the 
	performance of the controlled system. We suppose that $\nu$  is the L\'evy measure of a compound Poisson process with
	$\bar{\sigma} = \sqrt{\lambda\,\mathbb{E}[\xi^2]} = \sqrt{\lambda\,(\tau^2 + \mu^2)}$, 
	where $\lambda$ denotes the jump intensity and $\xi$ the jump size with mean $\mu$ and standard deviation $\tau$. Set $x= 0, \delta = 0.05, \beta = 1.0, H = 1.0,\sigma = 0.2, \mu= 0.0$

	\begin{figure}[htbp]
		
		\centering
		\includegraphics[width=0.47\textwidth]{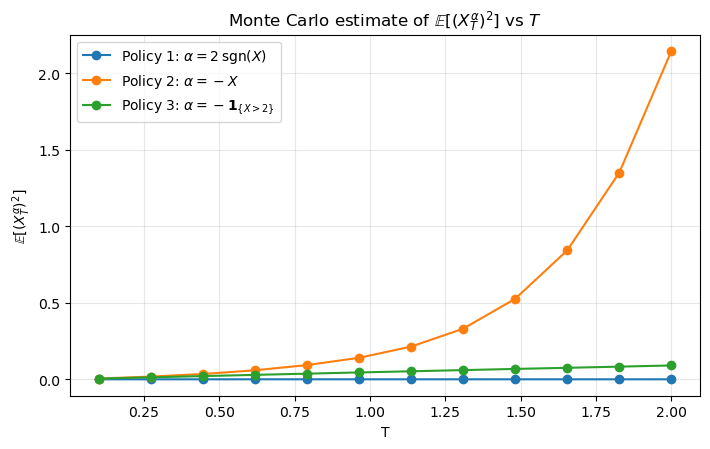}
		\caption{Dependence of $\mathbb{E}[(X_T^{\alpha})^2]$ on $T$ for different $\alpha$.}
		\label{fig:alpha}
	\end{figure}

	\begin{figure}[htbp]
		\centering
		\includegraphics[width=0.47\textwidth]{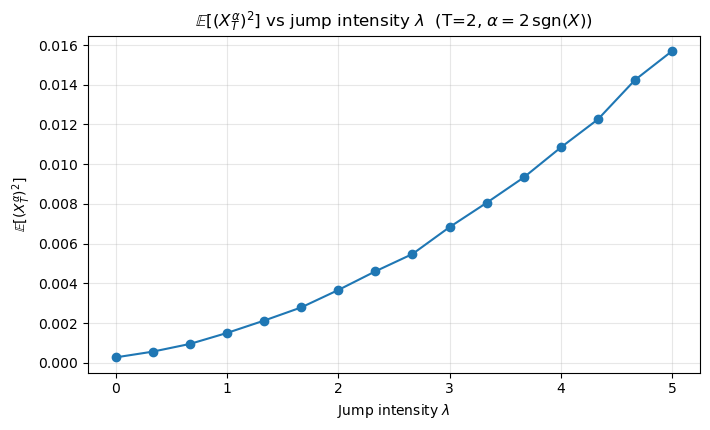}
		\includegraphics[width=0.47\textwidth]{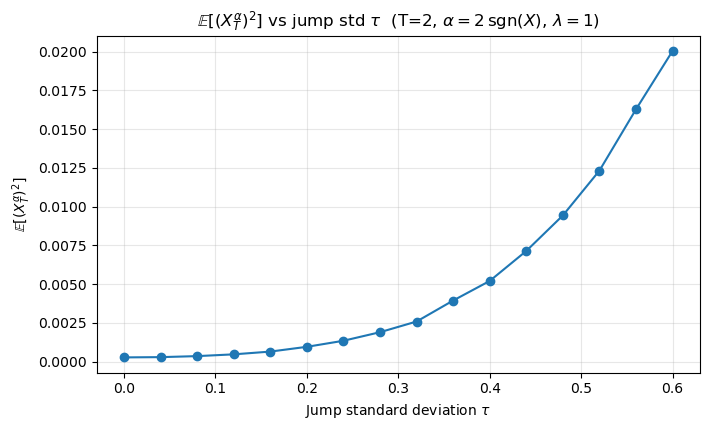}
		\caption{Dependence of $\mathbb{E}[(X_T^{\alpha})^2]$ on (left) jump intensity $\lambda$ and (right) 
			jump-size standard deviation $\tau$ for $T=2$ under the feedback control $\alpha_t = 2\,\mathrm{sgn}(X_t)$.}
		\label{fig:lambda_tau_dependence}
	\end{figure}

	\paragraph{Interpretation of Results.}
	
	Figure~\ref{fig:alpha} compares the terminal second moment\\ $\mathbb{E}[(X_T^{\alpha})^2]$ under three feedback controls: $\alpha_t=-X_t$, $\alpha_t=-\mathbf{1}_{\{X_t>2\}}$, and the sign-based rule $\hat{\alpha}_t=2\,\mathrm{sgn}(X_t)$ (optimal). The latter yields the smallest second moment, reflecting its stronger stabilising effect. It reacts symmetrically to deviations from the target, while the linear feedback provides weaker correction and the threshold policy performs worst due to limited intervention.
	
	Figure~\ref{fig:lambda_tau_dependence} shows that $\mathbb{E}[(X_T^{\alpha})^2]$ increases with both the claim intensity $\lambda$ and volatility $\tau$. Higher jump frequency or variability amplifies reserve fluctuations despite optimal control. This highlights the trade-off between control and jump risk, underscoring the need to account for claim statistics in control design.

	\section{Concluding remarks}

	In this paper, we have established a necessary and sufficient stochastic maximum principle for jump--diffusion systems with a piecewise Lipschitz drift coefficient. The analysis begins with an explicit representation of the first variation process, derived by combining It\^o’s formula for jump-diffusions with some stochastic integral estimates and the fact that the law of the SDE solution admits a density with respect to the Lebesgue measure.
	
	To handle the discontinuity in the drift, we constructed a sequence of approximating control problems with smooth coefficients. Ekeland’s variational principle was then applied to obtain a sequence of adjoint processes, allowing us to derive the maximum principle in the limit. The framework accommodates drift functions of the form $b(x,a) = b_1(x,a) + b_2(x) + b_3(x),$
	where the discontinuity arises from \( b_2 \).
	
	The present setting, however, does not include the case in which the insurer seeks to optimise a parameter \( \beta > 0 \) representing the magnitude of a fixed drain or injection in the model \eqref{eqmot1}--\eqref{eqmot2}. In that context, the drift takes the form $b(x,a) = b_1(x) + b_2(x)b_3(a),$ where \( b_2 \) is piecewise Lipschitz (and expressible as the difference of two monotone functions), and \( b_1 \) is continuously differentiable with bounded derivative and linear growth. We will consider this challenging topic in the future.
	
	\appendix

	\section{Proof of preliminary results on SDEs with jumps and irregular drifts}\label{appen}
	\begin{proof}[Proof of Lemma \ref{prop:convXn0}]
		To prove the Lemma, we proceed as in \cite{PrzSzo21, PrSX21}. We consider the function $G$ defined by:
		\begin{align}\label{eqdefG1}
			G(x)=x+\sum_{k=1}^m\alpha_k\phi(\frac{x-\xi_k}{c})(x-\xi_k)|x-\xi_k|,
		\end{align}
		with
		\begin{equation*}
			\phi(u)=\Big\{\begin{array}{lll}
				(1+u)^3(1-u)^3 & \text{ if } |u|\leq 1,\\
				0 & \text{ otherwise, }
			\end{array}
		\end{equation*} 
		$$
		\mathbb{R}_0 \ni \alpha_k=\frac{b(\xi_{k^-})-b(\xi_{k^+})}{2},\,\,\,k\in\{1,\ldots,m\}
		$$
		\begin{align}\label{eqc1}
			(0,\infty)\ni c<\min\big(\min_{1\leq k\leq m}\frac{1}{6|\alpha_k|},\min_{1\leq k\leq m-1}\frac{\xi_{k+1}-\xi_{k}}{2}\big).
		\end{align}
		
		It can be shown that one may choose $c$ such that $G'(x)> 0$ for all $x\in \mathbb{R}$. Hence $G$ admits a global inverse $G^{-1} : \mathbb{R}\mapsto\mathbb{R}$. Both $G$ and its inverse $G^{-1}$ are Lipschitz and $G \in C^1$, with bounded derivative. In particular, $G$ is Lipschitz since it is piecewise Lipschitz and continuous (see for example \cite{LeoSzo17}).
		
		Let $(G_n,b_{1,n},b_{2,n})_{n\ge1}$ be the approximations
		constructed in Lemma~\ref{lem:barb_bn_simultaneous},
		set $b_n:=b_{1,n}+b_{2,n}$, and define
		$Y^x:=G(X^x)$ and $Y^{n,x}:=G_n(X^{n,x})$.
		By the generalised It\^o formula for $G$ and the classical
		It\^o formula for $G_n$, these processes satisfy, respectively,
		\begin{equation}	\label{eq:SDErvmainZ1}
			\mathrm{d}Y^x_t = \bar b(Y^x_t)\mathrm{d}t + \bar \sigma(Y^x_t) \diffns B_t +\int_{\mathbb{R}_0}\bar \gamma (Y^x_{t-} ,z) N(\mathrm{d}z,\diffns t), t\in[0,T] , Y^x_0=G(x) ,
		\end{equation}
		and 
		\begin{equation}	\label{eq:SDErvmainZn1}
			\begin{cases}
				\mathrm{d}Y^{n,x}_t = \bar b_n(Y^{n,x}_t)\mathrm{d}t + \bar \sigma_n(Y^{n,x}_t) \diffns B_t +\int_{\mathbb{R}_0}\bar \gamma_n (Y^{n,x}_{t-} ,z) N(\mathrm{d}z,\diffns t), t\in[0,T]\\
				Y^{n,x}_0=G_n(x) ,
			\end{cases}
		\end{equation}
		where 
		\begin{equation}\label{eqtrans1}
			\begin{cases}
				\bar b(y) =G'(G^{-1} (y)) b(G^{-1} (y)) +\frac{1}{2} G''(G^{-1} (y)),\\
				\bar b_n(y) =G_n'(G_n^{-1} (y)) b_n(G_n^{-1} (y)) +\frac{1}{2} G_n''(G_n^{-1} (y)),\\
				\,\,\,\,\,\,\,\,\,\,\,\,\,\,\,\,=
				G_n'(G_n^{-1}(y))\,b_{1,n}(G_n^{-1}(y))
				+\bar b_{2,n}(y),	\quad	\bar\sigma(y)=G'(G^{-1} (y)),\\
				\bar\gamma(y,z) =G(G^{-1} (y)+\gamma (z))-G(G^{-1} (y))=G(G^{-1} (y)+\gamma (z))-y,\\
				\bar\sigma_n(y)=G_n'(G_n^{-1} (y)),	\quad\bar\gamma_n(y,z)= G_n(G_n^{-1} (y)+\gamma (z))-y.
			\end{cases}
		\end{equation}

		By \cite{LeoSzo17}, the transformed coefficients $\bar b$ and $\bar\sigma$ are globally Lipschitz. Moreover, since $G'$ is bounded and bounded away from zero, $\bar\sigma$ is bounded and uniformly elliptic. The approximations constructed in Lemma~\ref{lem:barb_bn_simultaneous} satisfy $b_n\to b$ Lebesgue-a.e., while $\bar b_n$ have uniformly bounded Lipschitz constants and satisfy a uniform linear-growth bound. By Remark~\ref{rem:transformed_coefficients_n}, $\bar\sigma_n\to\bar\sigma$ locally uniformly, and the coefficients $\bar\sigma_n$ are uniformly bounded, uniformly elliptic, and Lipschitz with a common constant. The same remark establishes the convergence and the uniform integrated Lipschitz and moment bounds for $\bar\gamma_n$.
		
		It follows from \cite[Theorem~6.2.9]{Apple2004} that \eqref{eq:SDErvmainZ1} and \eqref{eq:SDErvmainZn1} admit  unique global strong solutions. Furthermore, Theorem~\ref{thm:global_bound_clean} implies that, for every $t>0$, the laws of $Y_t^x$ and $Y_t^{n,x}$ admit densities, with density bounds for the approximating processes that are uniform in $n$. Moreover, one has:
		\begin{align}\label{convSDEY1}
			\mathbb{E}\big[\sup_{0\leq s\leq t}	|Y^x_t-Y^{n,x}_t|^2\big] \rightarrow 0 \text{ as }  n \rightarrow \infty
		\end{align}
		Indeed, using H\"older, B-D-G, Kunita and Gronwall's inequalities, we have
		\begin{align}\label{eqboundsupy}
			&\mathbb{E}\big[\sup_{0\leq t\leq T}	|Y^x_t-Y^{n,x}_t|^2\big]\notag\\
			\leq &16\big(|G(x)-G_n(x)|^2+  \mathbb{E}\big[\sup_{0\leq t\leq T} \big|\int_0^t\big\{\bar b(Y^x_s)-\bar b_n(Y^{n,x}_s)\big\}\mathrm{d}s\big|^2\big]\notag\\
			& +\mathbb{E}\big[\sup_{0\leq t\leq T}\big| \int_0^t\big\{\bar \sigma(Y^x_s) -\bar \sigma_n(Y^{n,x}_s) \big\}\diffns B_s\big|^2\big] \notag\\
			&+\mathbb{E}\big[\sup_{0\leq t\leq T}\big|\int_0^t\int_{\mathbb{R}_0}\big\{\bar \gamma (Y^x_{s-} ,z)-\bar \gamma_n (Y^{n,x}_{s-} ,z)\big\} N(\mathrm{d}z,\diffns s)\big|^2\big]\notag\\
			\leq &	C_T|G(x)-G_n(x)|^2+C_T \mathbb{E}\big[ \int_0^T \big|\bar b(Y^x_s)-\bar b_n(Y^x_s)+\bar b_n(Y^x_s)-\bar b_n(Y^{n,x}_s)\big|^2\mathrm{d}s\big]\notag\\
			&+C_T \mathbb{E}\big[ \int_0^T \big|\bar \sigma(Y^x_s)-\bar \sigma_n(Y^x_s)+\bar \sigma_n(Y^x_s)-\bar \sigma_n(Y^{n,x}_s)\big|^2\mathrm{d}s\big]\notag\\
			&+C_T\mathbb{E}\big[\int_0^T\int_{\mathbb{R}_0}\big|\bar \gamma (Y^x_s ,z)-\bar \gamma_n (Y^{x}_{s-} ,z)+\bar \gamma_n (Y^{x}_{s-} ,z)-\bar \gamma_n (Y^{n,x}_{s-} ,z)\big|^2 \nu(\mathrm{d}z)\diffns s\big]\notag\\
			\leq &C_T|G(x)-G_n(x)|^2+C_T \int_0^T \mathbb{E}\big[ \big|\bar b(Y^x_s)-\bar b_n(Y^{x}_s)\big|^2\big]\mathrm{d}s\notag\\
			&+C_T \mathbb{E}\big[ \int_0^T \big|\bar \sigma(Y^x_s)-\bar \sigma_n(Y^x_s)\big|^2\mathrm{d}s\big]	+C_T\mathbb{E}\big[\int_0^T\int_{\mathbb{R}_0}\big|\bar \gamma (Y^x_{s-} ,z)-\bar \gamma_n (Y^{x}_{s-} ,z)\big|^2 \nu(\mathrm{d}z)\diffns s\big],
		\end{align} 
		where $C$ is a constant that may change from one line to another and we have used the uniform Lipschitz continuity of $\bar b_n, \bar \sigma_n$ and $\bar \gamma_n$. Since the law of $Y_s^x$ admits a density $\rho_s^x$ for
		every $s>0$, Tonelli's theorem gives
		\[
		\int_0^T
		\mathbb E\left[
		|\bar b(Y_s^x)-\bar b_n(Y_s^x)|^2
		\right]\,\mathrm ds
		=
		\int_0^T\int_{\mathbb R}
		|\bar b(y)-\bar b_n(y)|^2
		\rho_s^x(y)\,\mathrm dy\,\mathrm ds.
		\]
		For each $s>0$, the almost-everywhere convergence
		$\bar b_n\to\bar b$ and the uniform linear-growth bound imply,
		by dominated convergence, that $
		\int_{\mathbb R}
		|\bar b(y)-\bar b_n(y)|^2\rho_s^x(y)\,\mathrm dy
		\longrightarrow0.$
		Indeed, the integrand is bounded by
		$C(1+|y|^2)\rho_s^x(y)$, which is integrable. Moreover,
		\[
		\mathbb E\left[
		|\bar b(Y_s^x)-\bar b_n(Y_s^x)|^2
		\right]
		\le C\left(1+\mathbb E|Y_s^x|^2\right)
		\le C_T(1+|x|^2).
		\]
		Applying dominated convergence once more, now in $s$, shows
		that the drift-error term in \eqref{eqboundsupy} tends to zero.
		
		The last two terms in \eqref{eqboundsupy} also converge to
		zero. For the diffusion term, this follows from the local
		uniform convergence $\bar\sigma_n\to\bar\sigma$ and their
		uniform boundedness. For the jump term, it follows from the
		pointwise convergence $\bar\gamma_n\to\bar\gamma$ and the bound $ |\bar\gamma_n(y,z)-\bar\gamma(y,z)|^2
		\le4\overline g^{\,2}|\gamma(z)|^2,$ together with \eqref{eq:jump_moments} and dominated convergence.
		Finally, $G_n(x)\to G(x)$. Thus \eqref{convSDEY1} is proved. 
		
		Finally, since $G_n=G*\varphi_{1/n}$ and $G$ is globally
		Lipschitz, define
		\[
		e_n:=\sup_{u\in\mathbb R}|G_n(u)-G(u)|
		\le
		\frac{\overline g}{n}
		\int_{\mathbb R}|v|\varphi(v)\,\mathrm dv
		\longrightarrow0.
		\]
		The uniform lower bound $G_n'\ge\underline g$ gives
		$ \sup_{y\in\mathbb R}	 |G_n^{-1}(y)-G^{-1}(y)|
		\le\underline g^{-1}e_n.$ Therefore,
		\begin{align}\label{convSDEX1}
			\mathbb E\big[
			\sup_{0\le s\le T}|X_s^x-X_s^{n,x}|^2
			\big]
			=&
			\mathbb E\big[
			\sup_{0\le s\le T}
			|G^{-1}(Y_s^x)-G_n^{-1}(Y_s^{n,x})|^2
			\big]\notag\\
			\le&
			2\underline g^{-2}
			\mathbb E\big[
			\sup_{0\le s\le T}|Y_s^x-Y_s^{n,x}|^2
			\big]
			+2\underline g^{-2}e_n^2
			\longrightarrow0.
		\end{align}
		
	\end{proof}
	
	\begin{remark}\label{remdenY1}
		$c$ in \eqref{eqc1} can be chosen small enough to ensure that 
		\begin{equation}\label{eq:G_assump_simul}
			0<\underline g:=\inf_{x\in\mathbb R}G'(x)
			\le \sup_{x\in\mathbb R}G'(x)=:\overline g<\infty,
			\qquad
			G''\in L^\infty(\mathbb R).
		\end{equation}
		As a consequence, $\bar\sigma=G'\circ G^{-1}$ is bounded and
		uniformly elliptic. Together with the Lipschitz and growth
		properties established above, it follows from Theorem \ref{thm:global_bound_clean} that the law of $Y_t^x$ has a density with respect to Lebesgue measure for every $t>0$. Since $G$ is a $C^1$-diffeomorphism, the same holds for $X_t^x=G^{-1}(Y_t^x)$. 
	\end{remark}

	\begin{proof}[Proof of Theorem \ref{thmmainresder1}]
		
		For each fixed $n$, the coefficients $\bar b_n$ and
		$\bar\sigma_n$ are smooth, and
		$\bar\gamma_n(\cdot,z)$ is smooth with respect to the state variable.
		Since the jump activity is finite, differentiating the
		diffusion flow between successive jump times, using once more 
		\cite[Chapter V, Theorem 39, p. 305]{Prot05},
		and differentiating the jump maps yields the first
		variation process $U^{n,x}:=\partial_xY^{n,x}$, with
		initial condition $U_0^{n,x}=G_n'(x)$.
\begin{align}\label{eq:SDErvmainderZ1}
	U_t^x
	={}&G'(x)
	+\int_0^t\bar b'(Y_s^x)U_s^x\,\mathrm ds
	+\int_0^t\bar\sigma'(Y_s^x)U_s^x\,\mathrm dB_s
	\notag\\
	&+\int_0^t\int_{\mathbb R_0}
	\partial_y\bar\gamma(Y_{s-}^x,z)U_{s-}^x
	\,N(\mathrm dz,\mathrm ds),
\end{align}
where $\bar b'$ and $\bar\sigma'$ denote bounded Borel
representatives of their weak derivatives.

The coefficients $\bar b'$, $\bar\sigma'$,
$\partial_y\bar\gamma$ and their approximations
$\bar b_n'$, $\bar\sigma_n'$, $\partial_y\bar\gamma_n$
are bounded, uniformly in $n$. Since $G'$ and $G_n'$
are also uniformly bounded, standard moment estimates
for linear SDEs with finite-activity jumps yield,
for every $p\ge2$,
\begin{align}\label{eqestfirstvarY1}
	\mathbb E\big[
	\sup_{0\le s\le T}|U_s^x|^p
	\big]
	+
	\sup_{n\ge1}\mathbb E\big[
	\sup_{0\le s\le T}|U_s^{n,x}|^p
	\big]
	\le C_{p,T}.
\end{align}
Using these bounds, the almost-everywhere convergence
of $\bar b_n'$ to $\bar b'$, and the convergence
properties of $\bar\sigma_n'$ and
$\partial_y\bar\gamma_n$ in
Remark~\ref{rem:transformed_coefficients_n},
we establish below that
\begin{equation}\label{eqstronYde1}
	U_t^{n,x}\longrightarrow U_t^x
	\quad\text{strongly in }L^2(\Omega,\mathbb P).
\end{equation}
		
		Indeed
		\begin{align*}
			&\big|U^{n,x}_t -U^x_t\big|^2 \\
			\leq&8|G'(x)-G_n'(x)|^2+8 \big|\int_0^t\big(\bar b_n'(Y^{n,x}_s)U^{n,x}_s-\bar b'(Y^x_s)U^{x}_t\big)\mathrm{d}s\big|^2 \\
			&+ 8\big|\int_0^t\big(\bar \sigma_n'(Y^{n,x}_s) U^{n,x}_s-\bar \sigma'(Y^x_s) U^{x}_s\big)\diffns B_s\big|^2\\ 
			&+8\big|\int_0^t\int_{\mathbb{R}_0}\big(\partial_y\bar \gamma_n (Y^{n,x}_{s-} ,z)U^{n,x}_{s-}-\partial_y\bar \gamma (Y^x_{s-} ,z)U^{x}_{s-}\big) N(\mathrm{d}z,\diffns s)\big|^2\\
			\leq&16\big(|G'(x)-G_n'(x)|^2\notag\\
			&+ \big|\int_0^t\big(\bar b_n'(Y^{n,x}_s)-\bar b'(Y^{x}_s)\big)U^{n,x}_s\mathrm{d}s\big|^2 +\big|\int_0^t\bar b'(Y^x_s)\big(U^{n,x}_s-U^{x}_s\big)\mathrm{d}s\big|^2 \\
			&+ \big|\int_0^t\big(\bar \sigma_n'(Y^{n,x}_s)-\bar \sigma'(Y^{x}_s)\big)U^{n,x}_s\diffns B_s\big|^2+ \big|\int_0^t \bar \sigma'(Y^x_s)\big(U^{n,x}_s-U^{x}_s\big) \diffns B_s\big|^2\\ &+\big|\int_0^t\int_{\mathbb{R}_0}\big(\partial_y\bar \gamma_n(Y^{n,x}_{s-},z)-\partial_y\bar \gamma(Y^{x}_{s-},z)\big)U^{n,x}_{s-} N(\mathrm{d}z,\diffns s)\big|^2\\
			&+\big|\int_0^t\int_{\mathbb{R}_0}\partial_y \bar \gamma(Y^x_{s-},z)\big(U^{n,x}_{s-}-U^{x}_{s-}\big) N(\mathrm{d}z,\diffns s)\big|^2\big)	
		\end{align*}
			\begin{align*}
			\leq&16\big(|G'(x)-G_n'(x)|^2\notag\\
			&+ \int_0^t\big|\bar b_n'(Y^{n,x}_s)-\bar b'(Y^{x}_s)\big|^2\mathrm{d}s\int_0^t\big|U^{n,x}_s\big|^2\mathrm{d}s +TM^2\int_0^t \big|U^{n,x}_s-U^{x}_s\big|^2\mathrm{d}s \\
			&+ \big|\int_0^t\big(\bar \sigma_n'(Y^{n,x}_s)-\bar \sigma'(Y^{x}_s)\big)U^{n,x}_s\diffns B_s\big|^2+ \big|\int_0^t \bar \sigma'(Y^x_s)\big(U^{n,x}_s-U^{x}_s\big) \diffns B_s\big|^2\\ &+\big|\int_0^t\int_{\mathbb{R}_0}\big(\partial_y\bar \gamma_n(Y^{n,x}_{s-},z)-\partial_y\bar \gamma(Y^{x}_{s-},z)\big)U^{n,x}_s N(\mathrm{d}z,\diffns s)\big|^2\\
			&+\big|\int_0^t\int_{\mathbb{R}_0}\partial_y \bar \gamma(Y^x_{s-},z)\big(U^{n,x}_{s-}-U^{x}_{s-}\big) N(\mathrm{d}z,\diffns s)\big|^2\big).
		\end{align*}
	
	Taking expectations, decomposing the Poisson integral
	into its compensated part and compensator, and applying
	It\^o's isometry and H\"older's inequality, we obtain
		\begin{align}\label{convderZ1}
			&\mathbb{E}\big[	\big|U^{n,x}_t-U^{x}_t\big|^2\big]\notag\\
			\leq&C_{T,M}\Big( \mathbb{E}\big[	\int_0^t\big|\bar b_n'(Y^{n,x}_s)-\bar b'(Y^{x}_s)\big|^4\mathrm{d}s\big]^{1/2} +\int_0^t \mathbb{E}\big[\big|U^{n,x}_s-U^{x}_s\big|^2\big]\mathrm{d}s \notag\\
			&+ \mathbb{E}\big[\int_0^t\big|\bar \sigma_n'(Y^{n,x}_s)-\bar \sigma'(Y^{x}_s)\big|^4\diffns s\big]^{1/2}+ \int_0^t  \mathbb{E}\big[\big|U^{n,x}_s-U^{x}_s\big|^2 \big]\diffns s\notag\\ 
			&+\mathbb{E}\big[\int_0^t\int_{\mathbb{R}_0}\big|\partial_y\bar \gamma_n(Y^{n,x}_{s-},z)-\partial_y\bar \gamma(Y^{x}_{s-},z)\big|^4\nu(\mathrm{d}z)\,\diffns s\big]^{1/2}\notag\\
			&+2\int_0^t\int_{\mathbb{R}_0}\mathbb{E}\big[ \big|U^{n,x}_{s-}-U^{x}_{s-}\big|^2\big]\nu(\mathrm{d}z)\diffns s+|G'(x)-G_n'(x)|^2\Big)\notag \\
			\leq&C_{T,M,\nu(\R_0)}\Big(|G'(x)-G_n'(x)|^2+ \mathbb{E}\big[	\int_0^T\big|\bar b_n'(Y^{n,x}_s)-\bar b'(Y^{x}_s)\big|^4\mathrm{d}s\big]^{1/2}\notag\\
			& + \big[\int_0^T\mathbb{E}\big|\bar \sigma_n'(Y^{n,x}_s)-\bar \sigma'(Y^{x}_s)\big|^4\diffns s\big]^{1/2}\notag\\
			&+\mathbb{E}\big[\int_0^T\int_{\mathbb{R}_0}\big|\partial_y\bar \gamma_n(Y^{n,x}_{s-},z)-\partial_y\bar \gamma(Y^{x}_{s-},z)\big|^4\nu(\mathrm{d}z)\,\diffns s\big]^{1/2}\Big).
		\end{align}
		
		Let us now show that the first term on the right side converges to $0$. We have
		\begin{align*}
			\int_0^T\mathbb{E}\big[	\big|\bar b_n'(Y^{n,x}_s)-\bar b'(Y^{x}_s)\big|^4\big]\diffns s\leq &8\int_0^T\mathbb{E}\big[	\big|\bar b_n'(Y^{n,x}_s)-\bar b'(Y^{n,x}_s)\big|^4\big]\diffns s\\
			&+8\int_0^T\mathbb{E}\big[	\big|\bar b'(Y^{n,x}_s)-\bar b'(Y^{x}_s)\big|^4\big]\diffns s\\
			=&8\int_0^T(I_{1,n}(t)+I_{2,n}(t))\,\diffns t.
		\end{align*}
	The uniform linear growth of the coefficients in the SDE
			\eqref{eq:SDErvmainZn1}, together with
			\cite[Theorem~3.3.1]{KunitaB19}, yields the uniform moment bounds
			required by condition (L2) of
			Lemma \ref{lem:time_integrated_bn_full}. The density estimate in
			Theorem~\ref{thm:global_bound_clean} then allows us to apply
			Lemma~\ref{lem:time_integrated_bn_full} and obtain 	$\int_0^T I_{1,n}(t)\mathrm{d}t\longrightarrow0
			\,\,\text{as }n\to\infty.$ To prove that $\lim\limits_{n\rightarrow0}\int_0^TI_{2,n}(t)\diffns t=0$, let $k\in \mathbb{N}$ and consider the function $\bar b_k'$ as in as in Lemma \ref{lem:barb_bn_simultaneous}, we have
		\begin{align}\label{eqI2nfory}
			\int_0^TI_{2,n}(t)\diffns t \leq &C\big(\int_0^T\mathbb{E}\big[	\big|\bar b'(Y^{n,x}_s)-\bar b_k'(Y^{n,x}_s)\big|^4\big]\diffns t +\int_0^T\mathbb{E}\big[	\big|\bar b_k'(Y^{n,x}_s)-\bar b_k'(Y^{x}_s)\big|^4\big]\diffns t \notag\\
			&+\int_0^T\mathbb{E}\big[	\big|\bar b_k'(Y^{x}_s)-\bar b'(Y^{x}_s)\big|^4\big]\diffns t =C\int_0^T(I^1_{2,n}+I^2_{2,n}+I^3_{2,n})\diffns t .
		\end{align}

		By the same Egorov argument as in the proof of
		Lemma~\ref{lem:time_integrated_bn_full}, using the uniform
		density and moment bounds, we obtain $\sup_{n\ge1}\int_0^T I^1_{2,n}(s)\,\mathrm ds\longrightarrow0,\,\,
		\int_0^T I^3_{2,n}(s)\,\mathrm ds\longrightarrow0
		\quad\text{as }k\to\infty.$ For each fixed $k$, the continuity and boundedness of $\bar b_k'$, the convergence of $Y_s^{n,x}$ to $Y_s^x$ in probability, and dominated convergence imply that
		$\int_0^T I^2_{2,n}(s)\,\mathrm ds\longrightarrow0
		\,\,\text{as }n\to\infty.$ Taking $n\to\infty$ first and then $k\to\infty$ in \eqref{eqI2nfory} proves that
		$\int_0^T I_{2,n}(s)\,\mathrm ds\to0$. The diffusion-derivative term is treated similarly, using
		the uniform boundedness and local $L^4$ convergence of
		$\bar\sigma_n'$ to $\bar\sigma'$.
		For the jump-derivative term, the local uniform convergence
		of $\partial_y\bar\gamma_n(\cdot,z)$, the continuity of
		$\partial_y\bar\gamma(\cdot,z)$, and their uniform boundedness,
		together with $\nu(\mathbb R_0)<\infty$, give convergence by
		dominated convergence and localisation.	Therefore, \eqref{convderZ1} implies $	\mathbb E[|U_t^{n,x}-U_t^x|^2]\longrightarrow0.$ Since the preceding moment bounds are uniform in the
		initial condition, dominated convergence yields
		$U_t^{n,\cdot}\to U_t^\cdot$ in
		$L^2(\Omega\times K)$ for every compact interval $K$.
		Together with the corresponding convergence of
		$Y_t^{n,\cdot}$ to $Y_t^\cdot$, closedness of the weak
		derivative implies that $Y_t^\cdot$ is weakly
		differentiable and $\frac{\partial Y_t^x}{\partial x}=U_t^x$ for Lebesgue-a.e. $x$. Consequently, $	\frac{\partial Y_t^{n,x}}{\partial x}\longrightarrow \frac{\partial Y_t^x}{\partial x}$ strongly in $L^2(\Omega)$for Lebesgue-a.e. initial condition $x$.
		
		Moreover, since $G^{-1}$ is Lipschitz, it follows using the Sobolev chain rule that $X_t^x=G^{-1}(Y_t^x)$ is weakly differentiable with respect to the initial condition and we have $
		\frac{\partial X_t^{x}}{\partial x}
		=(G^{-1})'(Y_t^x)\frac{\partial Y_t^{x}}{\partial x},
		\quad
		\partial_xX_t^{n,x}
		=(G_n^{-1})'(Y_t^{n,x})\frac{\partial Y_t^{n,x}}{\partial x}.
		$
		Thus,
		
		\begin{align*}
			\frac{\partial X_t^{n,x}}{\partial x}
			-\frac{\partial X_t^x}{\partial x}
			={}&(G^{-1})'(Y_t^{x})
			\big(
			\frac{\partial Y_t^{n,x}}{\partial x}
			-\frac{\partial Y_t^x}{\partial x}
			\big)\\
			&+\left(
			(G^{-1}_n)'(Y_t^{n,x})-(G^{-1})'(Y_t^x)
			\right)\frac{\partial Y_t^{n,x}}{\partial x}.
		\end{align*}
		The first term converges to zero in $L^2(\Omega)$ by the
		boundedness of $(G^{-1})'$ and the strong $L^2$ convergence
		of $\frac{\partial Y_t^{n,x}}{\partial x}$. By Lemma \ref{lem:barb_bn_simultaneous},
		$(G_n^{-1})'\to(G^{-1})'$ locally uniformly, with a common
		uniform bound. Combining with $Y_t^{n,x}\to Y_t^x$ in
		probability, we have $
		\mathbb E\big[\big|(G_n^{-1})'(Y_t^{n,x})-(G^{-1})'(Y_t^x)\big|^4
		\big]\longrightarrow0.$ using
		Cauchy-Schwarz inequality and the uniform fourth-moment bound for
		$\frac{\partial Y_t^{n,x}}{\partial x}$, it is easy to show that the second term also converges to zero in $L^2(\Omega)$. 	Hence, $	\frac{\partial X_t^{n,x}}{\partial x}
		\longrightarrow
		\frac{\partial X_t^x}{\partial x}$ strongly in  $L^2(\Omega,\mathbb{P}).$

		Now, let us consider $X^{n,x}_t$, the solution to the SDE \textcolor{blue}{\eqref{eq:SDErvmainapp1}}. As mentioned above, $X^{n,x}_t$ is differentiable with respect to the initial condition and its derivative satisfies the following linear ordinary differential equation
		\begin{align*}	
			\frac{\partial X^{x,n}_t}{\partial x} =& 1 + \int_0^t\big( b'_{1,n}(X^{x,n}_u)+b'_{2,n}(X^{x,n}_u)\big)\frac{\partial X^{x,n}_u}{\partial x}\mathrm{d}u\\
			=&\exp\big\{\int_0^t \big( b'_{1,n}(X^{x,n}_u)+b'_{2,n}(X^{x,n}_u)\big)\mathrm{d}u\big\}.
		\end{align*}
		
		Set
		$
		\widetilde b_{2,n}(x):=\int_0^x b_{2,n}(y)\,\mathrm dy,
		\,\,\,
		\widetilde b_2(x):=\int_0^x b_2(y)\,\mathrm dy.
		$
		Then by It\^o's formula for L\'evy processes,
		\begin{align*}
			&\tilde b_{2,n}(X^{n,x}_t)- \tilde b_{2,n}(x)\\
			=&	\int_0^tb_1(X^{n,x}_u)b_{2,n}(X^{n,x}_u)  \mathrm{d}u+	\int_0^tb^2_{2,n}(X^{n,x}_u) \mathrm{d}u+\int_0^tb_{2,n}(X^{n,x}_u) \mathrm{d}B_u	\notag\\
			&+\frac{1}{2}\int_0^tb'_{2,n}(X^{n,x}_u) \mathrm{d}u+\int_0^t\int_{\mathbb{R}_0}\big\{ \tilde b_{2,n}(X^{n,x}_{u-}+\gamma(z))-\tilde b_{2,n}(X^{n,x}_{u-})\big\} N(\mathrm{d}z,\mathrm{d}u).
		\end{align*}
		Substituting in the preceding inequality gives
		\begin{align*}	
			\frac{\partial X^{x,n}_t}{\partial x}=&\exp\Big\{\int_0^t b'_{1,n}(X^{x,n}_u)\mathrm{d}u+2\big(\tilde b_{2,n}(X^{n,x}_t)-\tilde b_{2,n}(x)- \int_0^tb^2_{2,n}(X^{n,x}_u) \mathrm{d}u\\
			&-	\int_0^tb_{1,n}(X^{n,x}_u)b_{2,n}(X^{n,x}_u)  \mathrm{d}u-\int_0^tb_{2,n}(X^{n,x}_u) \mathrm{d}B_u\\
			&-\int_0^t\int_{\mathbb{R}_0}\big\{ \tilde b_{2,n}(X^{n,x}_{u-}+\gamma(z))-\tilde b_{2,n}(X^{n,x}_{u-})\big\} N(\mathrm{d}z,\mathrm{d}u)\big)\Big\}=:e^{A_n}.
		\end{align*}
		Denote by $A$ the corresponding exponent on the
		right side of \eqref{eqfirstvarX1}.
		
		\medskip
		\noindent\textbf{Claim 1.} We claim that $e^{A_n}\to e^A$ strongly in $L^2(\Omega,\mathbb P)$
		
		\begin{proof}[Proof of \textbf{Claim 1,}]
			
			Set $\mathcal J_{1,n}:=e^{A_n}$ and $\mathcal J_2:=e^A$ and write
			\[
			A_n-A=\sum_{j=1}^7\mathcal I_{j,n},
			\]
			where
			\begin{align*}
				\mathcal I_{1,n}
				:=&\int_0^t
				[b_{1,n}'(X_u^{n,x})-b_1'(X_u^x)]\,\mathrm du,\quad
				\mathcal I_{2,n}:=-2[\tilde b_{2,n}(x)-\tilde b_2(x)],\\
				\mathcal I_{3,n}
				:=&2[\tilde b_{2,n}(X_t^{n,x})-\tilde b_2(X_t^x)],\quad
				\mathcal I_{4,n}
				:=-2\int_0^t
				[b_{2,n}(X_u^{n,x})-b_2(X_u^x)]\,\mathrm dB_u,\\
				\mathcal I_{5,n}
				:=&-2\int_0^t
				[b_{2,n}^2(X_u^{n,x})-b_2^2(X_u^x)]\,\mathrm du,\\
				\mathcal I_{6,n}
				:=&-2\int_0^t
				[b_1(X_u^{n,x})b_{2,n}(X_u^{n,x})
				-b_1(X_u^x)b_2(X_u^x)]\,\mathrm du,\\
				\mathcal I_{7,n}
				:=&-2\int_0^t\int_{\mathbb R_0}
				\big(
				\tilde b_{2,n}(X_{u-}^{n,x}+\gamma(z))
				-\tilde b_{2,n}(X_{u-}^{n,x})
				\\
				&-\tilde b_2(X_{u-}^x+\gamma(z))
				+\tilde b_2(X_{u-}^x)
				\big)\,N(\mathrm dz,\mathrm du),
			\end{align*}
			The inequality $|e^a-e^b|\le |a-b|(e^a+e^b)$
			and the Cauchy-Schwarz inequality, give
			\begin{align*}
				\mathbb E[|e^{A_n}-e^A|^2]
				\le&
				\mathbb E\big[
				|A_n-A|^2|e^{A_n}+e^A|^2
				\big]
				=
				\mathbb E\big[
				\big|\sum_{j=1}^7\mathcal I_{j,n}\big|^2
				|\mathcal J_{1,n}+\mathcal J_2|^2
				\big]\\
				\leq&
				\mathbb E\big[
				\big|\sum_{j=1}^7\mathcal I_{j,n}\big|^4
				\big]^{1/2}
				\mathbb E[|\mathcal J_{1,n}+\mathcal J_2|^4]^{1/2}\\
				\le&
				C\big(
				\sum_{j=1}^7
				\mathbb E[|\mathcal I_{j,n}|^4]^{1/2}
				\big)
				\big(
				\mathbb E[|\mathcal J_{1,n}|^4]^{1/2}
				+\mathbb E[|\mathcal J_2|^4]^{1/2}
				\big).
			\end{align*}
			
			The convergence follows if we can show 	$\mathbb E[|\mathcal I_{j,n}|^4]\to0$ for $j=1,\ldots,7$,\\
			$\underset{n}{\sup}\mathbb E[|\mathcal J_{1,n}|^4]<\infty$, and
			$\mathbb E[|\mathcal J_2|^4]<\infty$. We detail the estimates for $\mathcal I_{4,n}$ and
			$\mathcal I_{7,n}$; and also the uniform
			fourth-moment bound for $\mathcal J_{1,n}$.
			The remaining convergence estimates follow from the
			approximation properties and similar arguments, while
			the fourth-moment bound for $\mathcal J_2$ follows
			in the same way. For  $\mathcal{I}_{4,n}|^4$, using the Burkholder-Davis-Gundy and H\"older inequalities, we obtain
			\begin{align*}
				\mathbb{E}[|\mathcal{I}_{4,n}|^4]
				&\le C_T\int_0^t
				\mathbb{E}\big[
				|b_{2,n}(X_u^{n,x})-b_2(X_u^x)|^4
				\big]\,\mathrm{d}u\\
				\le &C_T\big(
				\int_0^t\mathbb{E}\big[
				|b_{2,n}(X_u^{n,x})-b_2(X_u^{n,x})|^4
				\big]\,\mathrm{d}u+
				\int_0^t\mathbb{E}\big[
				|b_2(X_u^{n,x})-b_2(X_u^x)|^4
				\big]\,\mathrm{d}u
				\big).
			\end{align*}
			The first integral converges to zero by
			Lemma \ref{lem:time_integrated_bn_full}. For the second integral, we use the same approximation	argument as in the proof of \eqref{eqI2nfory}.	Consequently,
			$\mathbb{E}[|\mathcal{I}_{4,n}|^4]\rightarrow 0$. 	To estimate $\mathcal I_{7,n}$, set the term $I_{7,n}$, set	$H_n(u,z):=
			\widetilde b_{2,n}(X_{u-}^{n,x}+\gamma(z))
			-\widetilde b_{2,n}(X_{u-}^{n,x})
			-\widetilde b_2(X_{u-}^x+\gamma(z))
			+\widetilde b_2(X_{u-}^x).$
			Using $N=\widetilde N+\nu\,\mathrm du$, the fourth-moment
			estimate for compensated Poisson integrals (see for example \cite[Proposition 2.6.1]{KunitaB19} or \cite[Theorem 4.4.23]{Apple2004}), H\"older's inequality,
			and the fundamental theorem of calculus, we obtain
			\begin{align*}
				&	\mathbb E|\mathcal I_{7,n}|^4\notag\\
				=&
				\mathbb E\big|
				\int_0^t\int_{\mathbb R_0}
				H_n(u,z)\,N(\mathrm dz,\mathrm du)
				\big|^4\\
				\le	&
				8\mathbb E\big|
				\int_0^t\int_{\mathbb R_0}
				H_n(u,z)\,\widetilde N(\mathrm dz,\mathrm du)
				\big|^4+
				8\mathbb E\big|
				\int_0^t\int_{\mathbb R_0}
				H_n(u,z)\,\nu(\mathrm dz)\,\mathrm du
				\big|^4\\
				\le&
				C\mathbb E\big[
				\big(
				\int_0^t\int_{\mathbb R_0}
				|H_n(u,z)|^2\,\nu(\mathrm dz)\,\mathrm du
				\big)^2
				\big]+
				C\int_0^t\int_{\mathbb R_0}
				\mathbb E|H_n(u,z)|^4\,\nu(\mathrm dz)\,\mathrm du\\
				&+
				8\mathbb E\big|
				\int_0^t\int_{\mathbb R_0}
				H_n(u,z)\,\nu(\mathrm dz)\,\mathrm du
				\big|^4\\
				\le&
				C_T\int_0^t\int_{\mathbb R_0}
				\mathbb E|H_n(u,z)|^4\,\nu(\mathrm dz)\,\mathrm du\\
				=&
				C_T\int_0^t\int_{\mathbb R_0}
				|\gamma(z)|^4
				\mathbb E\big|
				\int_0^1
				\left\{
				b_{2,n}(X_{u-}^{n,x}+\lambda\gamma(z))
				-b_2(X_{u-}^x+\lambda\gamma(z))
				\right\}\,\mathrm d\lambda
				\big|^4
				\nu(\mathrm dz)\,\mathrm du\\
				\le&
				C_T\int_0^t\int_{\mathbb R_0}\int_0^1
				\mathbb E\big[
				\big|
				b_{2,n}(X_{u-}^{n,x}+\lambda\gamma(z))
				-b_2(X_{u-}^x+\lambda\gamma(z))
				\big|^4
				\big]
				|\gamma(z)|^4\,\mathrm d\lambda\,
				\nu(\mathrm dz)\,\mathrm du\\
				\le&
				C_T\int_0^t\int_{\mathbb R_0}\int_0^1
				\mathbb E\big[
				\big|
				b_{2,n}(X_{u-}^{n,x}+\lambda\gamma(z))
				-b_2(X_{u-}^{n,x}+\lambda\gamma(z))
				\big|^4
				\big]
				|\gamma(z)|^4\,\mathrm d\lambda\,
				\nu(\mathrm dz)\,\mathrm du\\
				&+
				C_T\int_0^t\int_{\mathbb R_0}\int_0^1
				\mathbb E\big[
				\big|
				b_2(X_{u-}^{n,x}+\lambda\gamma(z))
				-b_2(X_{u-}^x+\lambda\gamma(z))
				\big|^4
				\big]
				|\gamma(z)|^4\,\mathrm d\lambda\,
				\nu(\mathrm dz)\,\mathrm du\\
				=:&C_T\bigl(\mathcal I_{7,n}^1+\mathcal I_{7,n}^2\bigr).
			\end{align*}
			Here $C_T$ is independent of $n$ and may depend on
			$\nu(\mathbb R_0)$. For each fixed $z\in\mathbb R_0$ and $\lambda\in[0,1]$, the time integral in $\mathcal I_{7,n}^1$ converges to zero by Lemma~\ref{lem:time_integrated_bn_full}, applied to the translated processes $X^{n,x}+\lambda\gamma(z)$. The convergence of $\mathcal I_{7,n}^2$ follows from the same smooth approximation argument as in \eqref{eqI2nfory}, applied
			to the translated processes $X^{n,x}+\lambda\gamma(z)$ and
			$X^x+\lambda\gamma(z)$. Finally, the uniform boundedness of $b_{2,n}$ and $b_2$ bounds both integrands by $C|\gamma(z)|^4$. Hence \eqref{eq:jump_moments} and dominated convergence in
			$(\lambda,z)$ give $\mathcal I_{7,n}^1+\mathcal I_{7,n}^2\longrightarrow0.$ Therefore, $\mathbb E|\mathcal I_{7,n}|^4\longrightarrow0$.
			
			We now show that  $\sup_{n\ge1}\mathbb E[|\mathcal J_{1,n}|^4]<\infty$. Choose $M\ge1$, independent of $n$, such that
			$
			\sup_{n\ge1}\|b_{2,n}\|_\infty\le M,
			\quad
			\|b_1'\|_\infty\le M,
			\quad
			|b_1(v)|+|b_n(v)|\le M(1+|v|),
			\quad v\in\mathbb R.
			$
			
			Using the uniform boundedness of $b_1'$ and $b_{2,n}$,
			the linear growth of $b_1$, the uniform Lipschitz continuity
			of $\tilde b_{2,n}$, and the Cauchy-Schwarz inequality,
			we obtain
			\begin{align}\label{eqderXn1}
				&	\mathbb E[|\mathcal J_{1,n}|^4]\notag\\
				=&\mathbb E\Big[\exp\Big\{
				4\int_0^t b_{1,n}'(X_u^{n,x})\,\mathrm du
				+8\big(\tilde b_{2,n}(X_t^{n,x})
				-\tilde b_{2,n}(x)\big)
				-8\int_0^t b_{2,n}^2(X_u^{n,x})\,\mathrm du\notag\\
				&
				-8\int_0^t b_{1,n}(X_u^{n,x})b_{2,n}(X_u^{n,x})\,\mathrm du
				-8\int_0^t b_{2,n}(X_u^{n,x})\,\mathrm dB_u\notag\\
				&\qquad
				-8\int_0^t\int_{\mathbb R_0}
				\big(
				\tilde b_{2,n}(X_{u-}^{n,x}+\gamma(z))
				-\tilde b_{2,n}(X_{u-}^{n,x})
				\big)\,N(\mathrm dz,\mathrm du)
				\Big\}\Big]\notag\\
				\leq&
				e^{4MT+8M|x|}
				\mathbb E\Big[\exp\Big\{
				8M|X_t^{n,x}|
				+8M^2\int_0^t(1+|X_u^{n,x}|)\,\mathrm du\notag\\
				&\hspace{4mm}
				-8\int_0^t b_{2,n}(X_u^{n,x})\,\mathrm dB_u
				+8MS_T
				\Big\}\Big]\notag\\
				\le&
				e^{4MT+8M|x|+8M^2T}
				\mathbb E\big[
				\exp\big\{
				cR_n+8MS_T
				-8\int_0^t b_{2,n}(X_u^{n,x})\,\mathrm dB_u
				\big\}
				\big]\notag\\
				\leq&
				e^{4MT+8M|x|+8M^2T}
				\mathbb E[e^{2cR_n+16MS_T}]^{1/2}
				\mathbb E\big[
				e^{-16\int_0^t b_{2,n}(X_u^{n,x})\,\mathrm dB_u}
				\big]^{1/2}\notag\\
				\le&
				C_{M,T,x}\,
				\mathbb E[e^{2cR_n+16MS_T}]^{1/2},
			\end{align}
			where $
			R_n:=\sup_{0\le u\le T}|X_u^{n,x}|,\,\,\,
			S_T:=\int_0^T\int_{\mathbb R_0}
			|\gamma(z)|\,N(\mathrm dz,\mathrm du),\,\,\,
			c:=8M(1+MT).$
			In the last inequality, we used the exponential martingale
			property and the bound
			\[
			\mathbb E\big[
			e^{-16\int_0^t b_{2,n}(X_u^{n,x})\,\mathrm dB_u}
			\big]
			\le e^{128M^2T}.
			\]

			By the uniform linear-growth bound and Gronwall's inequality,
			\[
			R_n\le e^{KT}\big(
			|x|+KT+|\sigma|\sup_{0\le u\le T}|B_u|+S_T
			\big).
			\]
			Therefore, using independence between $B$ and $N$, we have
			\begin{align*}
				\mathbb E[e^{2cR_n+16MS_T}]
				\le&
				e^{2ce^{KT}(|x|+KT)}
				\mathbb E\big[
				e^{2ce^{KT}|\sigma|\sup_{0\le u\le T}|B_u|}
				\big]\\
				&\times
				\exp\big\{
				T\int_{\mathbb R_0}
				(e^{a_*|\gamma(z)|}-1)\,\nu(\mathrm dz)
				\big\},
			\end{align*}
			where $	a_*:=2ce^{MT}+16M=16M\big\{(1+MT)e^{MT}+1\big\}.$
			The Brownian exponential moment is finite. Hence, provided
			$a_0\ge a_*$ in the exponential moment assumption \eqref{eq:jump_exp_moments}, the
			right side is finite and independent of $n$.
			This proves the required uniform fourth-moment bound
			and completes the proof of the claim.
		\end{proof}

		Uniqueness of the limit therefore yields $
		\frac{\partial X_t^x}{\partial x}=e^A,\,\,\mathbb P$-a.s.,
		for Lebesgue-a.e. initial condition $x$.
		This proves \eqref{eqfirstvarX1}. The moment bound \eqref{eqderX1} follows from
		\eqref{eqestfirstvarY1}, the Sobolev chain rule, and
		the boundedness of $(G^{-1})'$, with H\"older's inequality
		covering $1\le p<2$.
	\end{proof}

	\section{Auxiliary Results}

	\begin{lemm}[An approximation result]
		\label{lem:barb_bn_simultaneous}
		Let $b=b_1+b_2$, where $b_1$ is globally Lipschitz and
		$b_2$ is bounded and piecewise Lipschitz, with discontinuities
		contained in $D:=\{\xi_1,\ldots,\xi_m\}$.
		Let $G:\mathbb R\to\mathbb R$ be the $C^1$-diffeomorphism
		defined by \eqref{eqdefG1} and \eqref{eqc1}. In particular,
		$G-\mathrm{id}$ has compact support,  $
		0<\underline g\le G'(x)\le\overline g<\infty,
		\quad x\in\mathbb R,
		\quad G''\in L^\infty(\mathbb R).$
		Define
		\[
		\bar b_2(y)
		:=
		G'(G^{-1}(y))\,b_2(G^{-1}(y))
		+\frac12G''(G^{-1}(y)),
		\]
		with its continuous extension at $G(D)$, and assume that this
		extension is globally Lipschitz. Set
		\[
		\bar b(y)
		:=
		G'(G^{-1}(y))\,b_1(G^{-1}(y))+\bar b_2(y).
		\]
		
		Then there exist smooth diffeomorphisms $G_n$ and functions
		$\bar b_{2,n}\in C_0^\infty(\mathbb R)$ with the following
		properties:
		\begin{enumerate}
			\item[(i)] $G_n\to G$ in $C^1_{\mathrm{loc}}(\mathbb R)$,
			and, for every $r\in[1,\infty)$, $		G_n''\to G''
			$ in $L^r_{\mathrm{loc}}(\mathbb R)$ and Lebesgue-a.e.
			Moreover, $	\underline g\le G_n'\le\overline g,
			\quad \|G_n''\|_\infty\le\|G''\|_\infty.$
			The inverses satisfy $	G_n^{-1}\to G^{-1}$ in $C^1_{\mathrm{loc}}(\mathbb R),	\,\,\,
			\sup_{n\ge1}\|(G_n^{-1})'\|_\infty
			\le\underline g^{-1}.$
			
			\item[(ii)] The functions $\bar b_{2,n}$ are uniformly
			bounded and have uniformly bounded Lipschitz constants.
			For every $r\in[1,\infty)$,
			\[
			\bar b_{2,n}\to\bar b_2
			\quad\text{in }W^{1,r}_{\mathrm{loc}}(\mathbb R) \quad
			\text{(in particular, }
			\bar b_{2,n}'\to\bar b_2'
			\text{ in }L^r_{\mathrm{loc}}(\mathbb R)\text{)}.
			\]

			\item[(iii)] Define
			\begin{equation}\label{eq:def_bn_from_barbn_simul}
				b_{2,n}(x):=
				\frac{\bar b_{2,n}(G_n(x))-\frac12G_n''(x)}
				{G_n'(x)},
				\qquad
				b_{1,n}:=b_1*\varphi_{1/n},
				\qquad
				b_n:=b_{1,n}+b_{2,n},
			\end{equation}
			where $\varphi_{1/n}$ is the mollifier used in the proof.
			Then $b_{2,n}\in C_0^\infty(\mathbb R)$,
			$\sup_{n\ge1}\|b_{2,n}\|_\infty<\infty$, and
			$b_{2,n}\to b_2$ Lebesgue-a.e., in
			$L^r_{\mathrm{loc}}(\mathbb R)$ for every finite $r$,
			and uniformly on compact subsets of $\mathbb R\setminus D$.
			Moreover,		
			\[
			\sup_n\|b_{1,n}'\|_\infty\le\operatorname{Lip}(b_1),
			\qquad
			b_{1,n}\to b_1
			\quad\text{locally uniformly and in }
			W^{1,r}_{\mathrm{loc}}(\mathbb R)
			\]
			Each $b_n$ is smooth with bounded first and second
			derivatives, whose bounds may depend on $n$.
			There exists $K>0$, independent of $n$, such that
			\[
			|b_n(x)|+|b(x)|\le K(1+|x|),
			\qquad x\in\mathbb R,
			\]
			after choosing bounded values for $b_2$ at the points of $D$.
			Finally, $b_n\to b$ Lebesgue-a.e., in
			$L^r_{\mathrm{loc}}(\mathbb R)$ for every finite $r$,
			and uniformly on every compact subset of
			$\mathbb R\setminus D$.

			\item[(iv)] Define the full transformed drift by
			\[
			\bar b_n(y)
			:=
			G_n'(G_n^{-1}(y))\,b_{1,n}(G_n^{-1}(y))
			+\bar b_{2,n}(y).
			\]
			The functions $\bar b_n$ and $\bar b$ are globally
			Lipschitz, with $\sup_{n\ge1}\operatorname{Lip}(\bar b_n)<\infty,
			\quad
			|\bar b_n(y)|+|\bar b(y)|\le K(1+|y|).$ Moreover, for every $r\in[1,\infty)$,
			\[
			\bar b_n\to\bar b
			\quad\text{in }W^{1,r}_{\mathrm{loc}}(\mathbb R).
			\]
			The sequence can be chosen so that, in addition, $		\bar b_n'\to\bar b'	\quad\text{Lebesgue-a.e. on }\mathbb R.$ 	All derivatives of Lipschitz functions are understood
			in the weak sense.
		\end{enumerate}
	\end{lemm}

	\begin{proof}
		Let $\varphi\in C_0^\infty(\mathbb R)$ be a nonnegative
		symmetric mollifier supported in $[-1,1]$, with integral one.
		Set $\varphi_{1/n}(y):=n\varphi(ny),
		\,\,\,
		G_n:=G*\varphi_{1/n}.$
		Since $G\in C^1(\mathbb R)$ and its weak second derivative
		is bounded, standard properties of mollification give $G_n\to G
		\quad\text{in }C^1_{\mathrm{loc}}(\mathbb R),$	and, for every $r\in[1,\infty)$, $G_n''=G''*\varphi_{1/n}\to G''
		\quad\text{in }L^r_{\mathrm{loc}}(\mathbb R)$ and at every Lebesgue point of $G''.$
		Moreover, $G_n'=G'*\varphi_{1/n}$, so $\underline g\le G_n'\le\overline g,	\quad\|G_n''\|_\infty\le\|G''\|_\infty.$
		Thus each $G_n$ is strictly increasing and maps $\mathbb R$
		onto $\mathbb R$. Since $G_n$ is smooth and $G_n'>0$,
		it is a smooth diffeomorphism.
		
		For any compact interval $K$, the Lipschitz bound on
		$G_n^{-1}$ gives
		\[
		\sup_{y\in K}|G_n^{-1}(y)-G^{-1}(y)|
		\le
		\underline g^{-1}
		\sup_{u\in G^{-1}(K)}|G_n(u)-G(u)|
		\longrightarrow0.
		\]
		Using $	(G_n^{-1})'=\frac{1}{G_n'\circ G_n^{-1}},$	we also obtain local uniform convergence of the inverse
		derivatives and $
		\|(G_n^{-1})'\|_\infty\le\underline g^{-1}.$
		This proves (i). 
		
		Furthermore, since $G-\mathrm{id}$ has
		compact support and the mollifiers are supported in $[-1,1]$,
		the functions $G_n''$ have support in a common compact
		interval.
		
		Choose $\chi\in C_0^\infty(\mathbb R)$ such that
		$0\le\chi\le1$, $\chi=1$ on $[-1,1]$, and
		$\chi=0$ outside $[-2,2]$. Define
		\[
		\bar b_{2,n}(y)
		:=
		\chi(y/n)(\bar b_2*\varphi_{1/n})(y).
		\]
		The function $\bar b_2$ is bounded, since $	\|\bar b_2\|_\infty
		\le
		\overline g\|b_2\|_\infty+\frac12\|G''\|_\infty.$ Therefore,
		$\|\bar b_{2,n}\|_\infty\le\|\bar b_2\|_\infty.$ 
		Differentiating the cutoff gives $\bar b_{2,n}'
		=
		\frac1n\chi'(\cdot/n)(\bar b_2*\varphi_{1/n})
		+
		\chi(\cdot/n)(\bar b_2'*\varphi_{1/n}),$
		and hence $	\sup_{n\ge1}\|\bar b_{2,n}'\|_\infty
		\le
		\|\chi'\|_\infty\|\bar b_2\|_\infty
		+\operatorname{Lip}(\bar b_2)<\infty.$	On each fixed compact interval, the cutoff equals one for
		all sufficiently large $n$. Standard mollification therefore
		yields $\bar b_{2,n}\to\bar b_2
		\quad\text{in }W^{1,r}_{\mathrm{loc}}(\mathbb R),
		\quad 1\le r<\infty.$
		This proves (ii). In particular, the convergence is locally
		uniform by the one-dimensional Sobolev embedding.
		
		Define $b_{2,n}$ by \eqref{eq:def_bn_from_barbn_simul}.
		Since $G_n$ and $\bar b_{2,n}$ are smooth and
		$G_n'\ge\underline g$, the function $b_{2,n}$ is smooth.
		Both $G_n''$ and $\bar b_{2,n}\circ G_n$ have compact support,
		so $b_{2,n}\in C_0^\infty(\mathbb R)$. Moreover,
		\[
		\|b_{2,n}\|_\infty
		\le
		\underline g^{-1}
		\big(
		\|\bar b_2\|_\infty+\frac12\|G''\|_\infty
		\big).
		\]
		
		The local uniform convergence of $\bar b_{2,n}$ and $G_n$
		implies that $\bar b_{2,n}\circ G_n\to\bar b_2\circ G$ locally uniformly. Together with (i), this gives $b_{2,n}\to b_2$ almost	everywhere. Uniform boundedness then implies convergence in	$L^r_{\mathrm{loc}}(\mathbb R)$ for every finite $r$.
		
		For the particular transformation $G$, the function $G''$
		is continuous on $\mathbb R\setminus D$. Thus
		$G_n''\to G''$ uniformly on compact subsets of
		$\mathbb R\setminus D$. The defining formula for $b_{2,n}$
		therefore gives uniform convergence of $b_{2,n}$ to $b_2$
		on those compact sets.
		
		By standard properties of mollification, $b_{1,n}$ satisfies
		the bounds and convergence assertions in (iii).
		Since $b_{2,n}\in C_0^\infty(\mathbb R)$, each
		$b_n=b_{1,n}+b_{2,n}$ is smooth with bounded first and
		second derivatives. The uniform linear-growth bound
		and the remaining convergence assertions follow from
		the corresponding properties of $b_{1,n}$ and $b_{2,n}$.
		This proves (iii).

		We now consider the full transformed drifts. Write
		$u_n(y):=G_n^{-1}(y),\,\, u(y):=G^{-1}(y).$ The Sobolev chain and product rules give
		
		\begin{equation}\label{eq:full_transformed_derivative}
			\bar b_n'(y)
			=
			\frac{G_n''(u_n(y))}{G_n'(u_n(y))}
			b_{1,n}(u_n(y))
			+b_{1,n}'(u_n(y))
			+\bar b_{2,n}'(y).
		\end{equation}
		for every $y\in\mathbb R$, with the analogous identity
		holding for the weak derivative $\bar b'$ for
		Lebesgue-a.e. $y\in\mathbb R$.
		
		The terms involving $G_n''$ are uniformly bounded, since
		their supports in the original variable lie in a common
		compact interval, where $b_1$ is bounded. The other terms
		are uniformly bounded by the Lipschitz constant of $b_1$
		and the estimate in (ii). Hence $\sup_{n\ge1}\|\bar b_n'\|_\infty<\infty.$	The same argument applies to $\bar b$. Their uniform linear-growth bound follows from their defining formulas, the bounds on $G_n'$, and the uniform linear growth of
		$G_n^{-1}$. In particular, the latter follows from
		\[
		|G_n^{-1}(y)|
		\le
		\underline g^{-1}\bigl(|y|+|G_n(0)|\bigr),
		\]
		since $(G_n(0))_{n\ge1}$ is bounded.
		
		The local uniform convergence of $\bar b_n$ to $\bar b$
		follows from (i), (ii), and the continuity of $b_1$.
		To prove convergence of their weak derivatives, fix a compact
		interval $K$. The first term on the right-hand side of
		\eqref{eq:full_transformed_derivative} converges almost
		everywhere on $K$ to its counterpart for $\bar b$.
		Indeed, away from the finite set $G(D)$, this follows from
		the local uniform convergence of $G_n''$ away from $D$
		and the convergence of the inverses.
		Uniform boundedness then gives convergence in $L^r(K)$.

		For the second term, choose a compact interval $I$
		containing $u_n(K)$ for all sufficiently large $n$
		and $u(K)$. By change of variables,
		\[
		\|(b_{1,n}'-b_1')\circ u_n\|_{L^r(K)}
		\le
		\overline g^{1/r}
		\|b_{1,n}'-b_1'\|_{L^r(I)}
		\longrightarrow0.
		\]
		Moreover, approximation of $b_1'$ by continuous functions,
		together with the same change-of-variables bound, gives
		$b_1'\circ u_n\to b_1'\circ u$ in $L^r(K)$.
		Hence $b_{1,n}'\circ u_n\to b_1'\circ u$ in $L^r(K)$. Together with the convergence of $\bar b_{2,n}'$ from (ii),
		this proves $\bar b_n'\to\bar b'
		\,\,\,\text{in }L^r_{\mathrm{loc}}(\mathbb R).$	Consequently,
		$\bar b_n\to\bar b$ in $W^{1,r}_{\mathrm{loc}}(\mathbb R)$.
		
		Finally, the local $L^1$ convergence of $\bar b_n'$ allows
		us to select a diagonal subsequence converging to $\bar b'$
		almost everywhere on $\mathbb R$. Relabelling this
		subsequence preserves all the preceding properties and
		proves (iv).
	\end{proof}
	
	\begin{remark}\label{rem:transformed_coefficients_n}
		Let $(G_n)_{n\ge1}$ be as in
		Lemma~\ref{lem:barb_bn_simultaneous}, and define
		\[
		\bar\sigma_n(y):=G_n'\bigl(G_n^{-1}(y)\bigr),
		\qquad
		\bar\gamma_n(y,z)
		:=G_n\bigl(G_n^{-1}(y)+\gamma(z)\bigr)-y.
		\]
		Then $\bar\sigma_n\to\bar\sigma$ locally uniformly, and
		\[
		\underline g\le\bar\sigma_n(y)\le\overline g,
		\quad
		|\bar\sigma_n(y)-\bar\sigma_n(v)|
		\le\frac{\|G''\|_\infty}{\underline g}|y-v|,\quad \text{uniformly in} \quad n
		\]

		Under \eqref{eq:jump_moments}, for every $p\ge2$ there exists
		$C_p<\infty$, independent of $n$, such that
		\[
		\int_{\mathbb R_0}|\bar\gamma_n(y,z)|^p\,\nu(\mathrm dz)
		\le C_p,
		\qquad
		\int_{\mathbb R_0}
		|\bar\gamma_n(y,z)-\bar\gamma_n(v,z)|^p\,\nu(\mathrm dz)
		\le C_p|y-v|^p.
		\]
		The same bounds hold for $\bar\gamma$. Moreover,
		$\bar\gamma_n(\cdot,z)\to\bar\gamma(\cdot,z)$ locally uniformly
		for each $z\in\mathbb R_0$, and dominated convergence gives
		\[
		\int_{\mathbb R_0}
		\sup_{y\in K}
		|\bar\gamma_n(y,z)-\bar\gamma(y,z)|^p\,\nu(\mathrm dz)
		\longrightarrow0 \quad \text{for every compact interval}\quad K\subset\mathbb R.
		\]
		
	\end{remark}

	\begin{thm}\label{thm:global_bound_clean}
		Let $B$ be a Brownian motion and let $N(\mathrm dz,\mathrm dt)$ be
		an independent Poisson random measure on
		$\mathbb R_0\times[0,\infty)$, with intensity
		$\nu(\mathrm dz)\mathrm dt$. Assume that $\lambda:=\nu(\mathbb R_0)<\infty.$
		Consider the SDE
		\[
		\mathrm dY_t^x
		=b(Y_t^x)\,\mathrm dt+\sigma(Y_t^x)\,\mathrm dB_t
		+\int_{\mathbb R_0}\gamma(Y_{t-}^x,z)\,N(\mathrm dz,\mathrm dt),
		\quad Y_0^x=x.
		\]
		
		Suppose that $b,\sigma:\mathbb R\to\mathbb R$ are globally
		Lipschitz and that there exist constants
		$0<\underline\sigma\le\overline\sigma<\infty$ such that $
		\underline\sigma\le |\sigma(x)|\le\overline\sigma,
		\,\, x\in\mathbb R.$
		Assume also that $\gamma:\mathbb R\times\mathbb R_0\to\mathbb R$
		is measurable and that, for some $K>0$ and $x,y\in\mathbb R$
		\begin{align}\label{eqLipgam}
			\int_{\mathbb R_0}
			|\gamma(x,z)-\gamma(y,z)|^2\,\nu(\mathrm dz)
			\le K|x-y|^2, \text{ and }
			\int_{\mathbb R_0}|\gamma(x,z)|^2\,\nu(\mathrm dz)
			\le K(1+|x|^2).
		\end{align}
		
		Then, for every $t>0$ and $x\in\mathbb R$, the law of $Y_t^x$
		admits a density. Moreover, for every $T>0$, these densities can
		be chosen so that
		\begin{equation}\label{denbound1}
			\sup_{x\in\mathbb R}\sup_{y\in\mathbb R}\rho_t^x(y)
			\le \frac{C_T}{\sqrt t},
			\qquad t\in(0,T],
		\end{equation}
		for some constant $C_T<\infty$.
	\end{thm}
	
	\begin{proof}
		First consider the diffusion without jumps,
		\[
		\mathrm dZ_s=b(Z_s)\,\mathrm ds+\sigma(Z_s)\,\mathrm dB_s.
		\]
		Let $p(s,u,v)$ denote its transition density. Since $b$ is
		globally Lipschitz and $\sigma$ is globally Lipschitz, bounded
		and uniformly elliptic, \cite[Theorem 1.2(i)]{MePeZh21}
		implies that, for every $T>0$, there exists $A_T<\infty$ such that
		\begin{equation}\label{eq:aronson_clean}
			p(s,u,v)\le \frac{A_T}{\sqrt s},
			\qquad 0<s\le T,\quad u,v\in\mathbb R.
		\end{equation}
		
		Fix $T>0$ and $t\in(0,T]$. Write
		\[
		N_t:=N(\mathbb R_0\times(0,t])
		\]
		for the number of jumps up to time $t$, and let
		$0<\tau_1<\cdots<\tau_{N_t}<t$ be their ordered times on
		$\{N_t\ge1\}$. Define
		\[
		S_t:=
		\begin{cases}
			\tau_{N_t}, & N_t\ge1,\\
			0, & N_t=0,
		\end{cases}
		\qquad
		\Delta_t:=t-S_t.
		\]
		Since the jump activity is finite and there is almost surely no	jump at the deterministic time $t$, we have $\Delta_t>0$ almost surely. By definition of $S_t$, there are no jumps on $(S_t,t]$. Hence, on this interval, $Y^x$ evolves as the diffusion with coefficients
		$(b,\sigma)$. Let
		\[
		\mathcal J_t
		:=
		\sigma\bigl(
		N(A\times(0,s]):
		A\in\mathcal B(\mathbb R_0),\ 0\le s\le t
		\bigr),
		\]
		and define
		\[
		\mathcal G_t
		:=
		\mathcal J_t
		\vee
		\sigma\bigl(B_{s\wedge S_t}:s\ge0\bigr).
		\]
		
		Thus, $\mathcal G_t$ contains the jump times and marks up to $t$
		and the Brownian path up to $S_t$.
		Observe that $S_t$ is $\mathcal J_t$-measurable. In addition, since $B$ is independent of $N$,
		conditionally on $\mathcal J_t$, the increments of $B$ after
		$S_t$ are independent of the Brownian path up to $S_t$.
		Therefore, conditionally on $\mathcal G_t$, the terminal value
		$Y_t^x$ has the same law as the diffusion started from
		$Y_{S_t}^x$ and run for time $\Delta_t$. Its conditional density
		is therefore
		\[
		\rho_t^{x\,|\,\mathcal G_t}(y)
		=
		p(\Delta_t,Y_{S_t}^x,y).
		\]
		Using \eqref{eq:aronson_clean}, we obtain
		\[
		\sup_{y\in\mathbb R}\rho_t^{x\,|\,\mathcal G_t}(y)
		\le
		\sup_{u,v\in\mathbb R}p(\Delta_t,u,v)
		\le
		\frac{A_T}{\sqrt{\Delta_t}}.
		\]
		Taking expectations and applying Tonelli's theorem, we see that
		the law of $Y_t^x$ admits the density
		\[
		\rho_t^x(y)
		:=
		\mathbb E\big[
		\rho_t^{x\,|\,\mathcal G_t}(y)
		\big].
		\]
		Moreover, for all $x,y\in\mathbb R$,
		\begin{equation}\label{eq:rho_by_gap}
			\rho_t^x(y)
			\le A_T\,\mathbb E\big[\Delta_t^{-1/2}\big].
		\end{equation}
		It remains to estimate $\mathbb E[\Delta_t^{-1/2}]$ uniformly over
		$t\in(0,T]$. Decompose with respect to $N_t$. If $N_t=0$, then $\Delta_t=t$.
		If $N_t=n\ge1$, then $\Delta_t=t-\tau_n$. Conditional on
		$N_t=n$, $\tau_n$ has the same distribution as the maximum of
		$n$ independent $\mathrm{Unif}(0,t)$ random variables.
		Equivalently, the conditional law of $\tau_n/t$ given $N_t=n$
		is $\mathrm{Beta}(n,1)$. Therefore,
		\[
		\begin{aligned}
			\mathbb E[(t-\tau_n)^{-1/2}\mid N_t=n]
			=\frac{n}{\sqrt t}\int_0^1
			(1-u)^{-1/2}u^{n-1}\,\mathrm du=\frac{n}{\sqrt t}\,
			B\!\left(n,\tfrac12\right).
		\end{aligned}
		\]
		Hence
		\[
		\mathbb E[\Delta_t^{-1/2}]
		=
		\frac{1}{\sqrt t}\,\mathbb P(N_t=0)
		+
		\frac{1}{\sqrt t}
		\sum_{n=1}^{\infty}
		\mathbb P(N_t=n)\,
		nB\!\left(n,\tfrac12\right).
		\]
		Using
		\[
		nB\!\left(n,\tfrac12\right)
		=
		\frac{\sqrt\pi\,\Gamma(n+1)}{\Gamma(n+\tfrac12)}
		\le C\sqrt n,
		\qquad n\ge1,
		\]
		we obtain
		\[
		\begin{aligned}
			\sum_{n=1}^{\infty}
			\mathbb P(N_t=n)\,nB\!\left(n,\tfrac12\right)
			&\le
			C\sum_{n=1}^{\infty}\mathbb P(N_t=n)\sqrt n=C\,\mathbb E[\sqrt{N_t}].
		\end{aligned}
		\]
		Since $N_t\sim\mathrm{Poisson}(\lambda t)$, Jensen's inequality
		gives
		\[
		\mathbb E[\sqrt{N_t}]
		\le \sqrt{\mathbb E[N_t]}
		=\sqrt{\lambda t}
		\le \sqrt{\lambda T}.
		\]
		Therefore there exists $C_T'<\infty$ such that
		\[
		\mathbb E[\Delta_t^{-1/2}]
		\le \frac{C_T'}{\sqrt t},
		\qquad t\in(0,T].
		\]
		Combining this estimate with \eqref{eq:rho_by_gap} yields
		\[
		\rho_t^x(y)
		\le A_T\frac{C_T'}{\sqrt t}
		=:\frac{C_T}{\sqrt t},
		\qquad t\in(0,T],\quad x,y\in\mathbb R.
		\]
		Taking the supremum over $x,y\in\mathbb R$ completes the proof.
	\end{proof}

	\begin{lemm}[A convergence result]
		\label{lem:time_integrated_bn_full}
		Fix $T>0$, $x\in\mathbb R$, and $p\ge1$. Let
		$(Y^{n,x})_{n\ge1}$ be jointly measurable processes, and let
		$f_n,f:\mathbb R\to\mathbb R$ be Borel measurable functions.
		Assume that:
		\begin{enumerate}
			\item[(L1)] $f_n\to f$ Lebesgue-a.e.\ on $\mathbb R$, and
			there exists $K>0$, independent of $n$, such that
			\[
			|f_n(y)|+|f(y)|\le K(1+|y|),
			\qquad y\in\mathbb R,\quad n\ge1.
			\]
			
			\item[(L2)] The processes satisfy the uniform moment bound
			\[
			\sup_{n\ge1}\sup_{t\in[0,T]}
			\mathbb E\bigl[|Y_t^{n,x}|^{2p}\bigr]<\infty.
			\]
			
			\item[(L3)] For every $t\in(0,T]$ and $n\ge1$, the law of
			$Y_t^{n,x}$ admits a density $\rho_t^{n,x}$, and there exists
			$C_T<\infty$, independent of $n$ and $t$, such that
			\begin{equation}\label{eq:rho_sup_bound_time_full}
				\sup_{n\ge1}
				\|\rho_t^{n,x}\|_{L^\infty(\mathbb R)}
				\le \frac{C_T}{\sqrt t},
				\qquad t\in(0,T].
			\end{equation}
		\end{enumerate}
		Then
		\[
		\lim_{n\to\infty}
		\int_0^T
		\mathbb E\big[
		|f_n(Y_t^{n,x})-f(Y_t^{n,x})|^p
		\big]\,\mathrm dt
		=0.
		\]
	\end{lemm}
	
	\begin{proof}
		Define
		\[
		I_n(t):=
		\mathbb E\!\left[
		|f_n(Y_t^{n,x})-f(Y_t^{n,x})|^p
		\right],
		\qquad t\in[0,T].
		\]
		
		By the uniform linear-growth condition, there exists $C_p>0$
		such that
		\[
		|f_n(y)-f(y)|^p
		\le C_p(1+|y|^p),
		\qquad y\in\mathbb R,\quad n\ge1.
		\]
		Assumption (L2) implies a uniform $p$-moment bound.
		Therefore,
		\begin{equation}\label{eq:domination_time_full}
			I_n(t)
			\le
			C_p\big(
			1+\sup_{n\ge1}\sup_{s\in[0,T]}
			\mathbb E|Y_s^{n,x}|^p
			\big)
			=:M_{p,T}<\infty.
		\end{equation}
		
		Fix $\varepsilon>0$. Choose $\theta\in(0,T)$ sufficiently small
		that $\theta M_{p,T}\le\varepsilon/2$. Then
		\begin{equation}\label{eq:split_time_int_full}
			\begin{aligned}
				\int_0^T I_n(t)\,\mathrm dt
				=
				\int_0^\theta I_n(t)\,\mathrm dt
				+
				\int_\theta^T I_n(t)\,\mathrm dt\le
				\frac{\varepsilon}{2}
				+
				\int_\theta^T I_n(t)\,\mathrm dt.
			\end{aligned}
		\end{equation}
		
		Next, let $\eta>0$. By the Cauchy--Schwarz and Markov
		inequalities,
		\[
		\begin{aligned}
			\mathbb E\big[
			(1+|Y_t^{n,x}|^p)
			\mathbf 1_{\{|Y_t^{n,x}|>R\}}
			\big]
			\le&
			\big(
			\mathbb E[(1+|Y_t^{n,x}|^p)^2]
			\big)^{1/2}
			\mathbb P(|Y_t^{n,x}|>R)^{1/2}\\
			\leq&
			\big(
			2+2\mathbb E|Y_t^{n,x}|^{2p}
			\big)^{1/2}
			\frac{(\mathbb E|Y_t^{n,x}|^p)^{1/2}}{R^{p/2}}.
		\end{aligned}
		\]
		Thus, by (L2), we can choose $R\ge1$ sufficiently large such that
		\begin{equation}\label{eq:R_choice_full}
			\sup_{n\ge1}\sup_{t\in[0,T]}
			\mathbb E\left[
			(1+|Y_t^{n,x}|^p)
			\mathbf 1_{\{|Y_t^{n,x}|>R\}}
			\right]
			\le\eta.
		\end{equation}
		The uniform linear-growth bound therefore gives
		\begin{align}
			I_n(t)
			\le&
			\mathbb E\big[
			|f_n(Y_t^{n,x})-f(Y_t^{n,x})|^p
			\mathbf 1_{\{|Y_t^{n,x}|\le R\}}
			\big]+
			C_p\mathbb E\big[
			(1+|Y_t^{n,x}|^p)
			\mathbf 1_{\{|Y_t^{n,x}|>R\}}
			\big]\nonumber\\
			\le&
			\mathbb E\big[
			|f_n(Y_t^{n,x})-f(Y_t^{n,x})|^p
			\mathbf 1_{\{|Y_t^{n,x}|\le R\}}
			\big]
			+C_p\eta.
			\label{eq:split_space_full}
		\end{align}
		
		By Egorov's theorem, for every $\delta_1>0$, there exists a
		Borel set $F\subset[-R,R]$ with
		$\mathrm{Leb}(F)<\delta_1$ such that $f_n\to f$ uniformly on
		$F_R^c:=[-R,R]\setminus F$. In particular,
		\[
		a_n:=
		\sup_{y\in F_R^c}|f_n(y)-f(y)|^p
		\longrightarrow0.
		\]
		We may take $\delta_1<2R$, so that $F_R^c$ is nonempty.
		Then
		\[
		\begin{aligned}
			&\mathbb E\big[
			|f_n(Y_t^{n,x})-f(Y_t^{n,x})|^p
			\mathbf 1_{\{|Y_t^{n,x}|\le R\}}
			\big]\\
			=&
			\mathbb E\big[
			|f_n(Y_t^{n,x})-f(Y_t^{n,x})|^p
			\mathbf 1_{\{Y_t^{n,x}\in F_R^c\}}
			\big]+
			\mathbb E\left[
			|f_n(Y_t^{n,x})-f(Y_t^{n,x})|^p
			\mathbf 1_{\{Y_t^{n,x}\in F\}}
			\right]\\
			\le&
			a_n+C_p(1+R^p)\mathbb P(Y_t^{n,x}\in F).
		\end{aligned}
		\]
		By (L3), for every $t\in[\theta,T]$,
		\[
		\begin{aligned}
			\mathbb P(Y_t^{n,x}\in F)
			=&\int_F\rho_t^{n,x}(y)\,\mathrm dy\le
			\|\rho_t^{n,x}\|_{L^\infty(\mathbb R)}
			\,\mathrm{Leb}(F)\le
			\frac{C_T}{\sqrt\theta}\,\delta_1.
		\end{aligned}
		\]
		Combining these estimates with \eqref{eq:split_space_full}
		yields
		\[
		I_n(t)
		\le
		a_n+
		\frac{C_pC_T(1+R^p)}{\sqrt\theta}\,\delta_1
		+C_p\eta,
		\qquad t\in[\theta,T].
		\]
		For the chosen $\eta$ and the corresponding $R$, choose
		$\delta_1>0$ sufficiently small such that $
		\frac{C_pC_T(1+R^p)}{\sqrt\theta}\,\delta_1\le\eta.$ Since $a_n\to0$, it follows that
		\[
		\limsup_{n\to\infty}I_n(t)\le(C_p+1)\eta,
		\qquad t\in[\theta,T].
		\]
		Letting $\eta\downarrow0$, we conclude that $
		I_n(t)\longrightarrow0,	\,\, t\in[\theta,T].$
		Finally, \eqref{eq:domination_time_full} and the dominated
		convergence theorem imply that $
		\int_\theta^T I_n(t)\,\mathrm dt\longrightarrow0.$
		Together with \eqref{eq:split_time_int_full}, this gives
		\[
		\limsup_{n\to\infty}\int_0^T I_n(t)\,\mathrm dt
		\le\frac{\varepsilon}{2}.
		\]
		Since $\varepsilon>0$ is arbitrary, the result follows.
	\end{proof}

	\begin{remark}\label{remcongamma}
		By Remark~\ref{rem:transformed_coefficients_n}, a decomposition
		into $\{|Y_t^{n,x}|\le R\}$ and its complement, together with
		the uniform moment bound, yields, for every $p\ge2$,
		\[
		\int_0^T\int_{\mathbb R_0}
		\mathbb E\!\left[
		|\bar\gamma_n(Y_t^{n,x},z)
		-\bar\gamma(Y_t^{n,x},z)|^p
		\right]\nu(\mathrm dz)\,\mathrm dt
		\longrightarrow0.
		\]
	\end{remark}



\end{document}